\documentclass[a4paper, reqno]{amsart}

\usepackage[english]{babel}
\usepackage[T1]{fontenc}
\usepackage[utf8]{inputenc}
\usepackage{amssymb}
\usepackage{amsmath}
\usepackage{amsfonts}
\usepackage{amsthm}
\usepackage[scr = dutchcal]{mathalfa}
\usepackage{dsfont}
\usepackage[numbers]{natbib}
\usepackage{url}
\usepackage{enumitem}
\usepackage{tikz}
\usetikzlibrary{patterns}

\newcommand{\N}{\mathds{N}}
\newcommand{\Z}{\mathds{Z}}

\newcommand{\R}{\mathds{R}}
\newcommand{\C}{\mathds{C}}

\providecommand{\abs}[1]{\left\lvert#1\right\rvert}
\providecommand{\norm}[1]{\left\lVert#1\right\rVert}
\providecommand{\differential}{\mathrm{d}}
\providecommand{\Hol}{\mathcal{H}}
\providecommand{\Mer}{\mathcal{M}\mathscr{er}}

\DeclareMathOperator{\TextRe}{Re}
\DeclareMathOperator{\TextIm}{Im}
\DeclareMathOperator{\Sec}{Sect}

\renewcommand{\d}{\differential}
\renewcommand{\Re}{\TextRe}
\renewcommand{\Im}{\TextIm}
\renewcommand{\exp}{\mathrm{e}}
\newcommand{\from}{\colon}

\newtheorem{lemma}{Lemma}[section]
\newtheorem{prop}[lemma]{Proposition}
\newtheorem{thm}[lemma]{Theorem}
\newtheorem{corollary}[lemma]{Corollary}

\theoremstyle{definition}
\newtheorem{defi}[lemma]{Definition}
\newtheorem{remark}[lemma]{Remark}
\newtheorem{example}[lemma]{Example}

\newcommand\rlim{
\mathchoice{\vcenter{\hbox{${\scriptstyle{+}}$}}}
{\vcenter{\hbox{$\scriptstyle{+}$}}}
{\vcenter{\hbox{$\scriptscriptstyle{+}$}}}
{\vcenter{\hbox{$\scriptscriptstyle{+}$}}}}

\begin{document}

\title[On the harmonic extension approach]
      {On the harmonic extension approach to fractional powers in Banach spaces}
\author[J.~Meichsner]{Jan Meichsner}
\author[C.~Seifert]{Christian Seifert}
\address{TU Hamburg \\ Institut f\"ur Mathematik \\
Am Schwarzenberg-Campus~3 \\
Geb\"aude E \\
21073 Hamburg \\
Germany}
\email{jan.meichsner@tuhh.de, christian.seifert@tuhh.de}

\subjclass[2010]{Primary 47A05, Secondary 47D06, 47A60}

\keywords{fractional powers, non-negative operator, Dirichlet-to-Neumann operator}

\date{\today}

\begin{abstract}
We show that fractional powers of general sectorial operators on Banach spaces can be obtained by the harmonic extension approach. 
Moreover, for the corresponding second order ordinary differential equation with incomplete data describing the harmonic extension we prove existence and uniqueness of a bounded solution 
(i.e.\ of the harmonic extension). 
\end{abstract}

\maketitle

\section{Introduction}

Fractional powers of closed linear operators
%, i.e.\ operators of the form $A^{\alpha}$ for given $\alpha \in \C$, $A$ a closed linear operator, both subject to further contraints, 
in Banach spaces are a classical topic in operator theory. Given a closed linear operator $A$ in a Banach space $X$ and $\alpha\in\C$ (in particular for $\alpha \in (0,1)$), 
one may ask whether we can define an operator $A^\alpha$, the fractional power of $A$. There are various different methods on how to do this for particular subclasses of 
operators.

To the best of our knowledge, the first method appeared in \cite{bochner1949}, where Bochner introduced subordination for generators of stochastic processes, see also 
\cite{feller1952}.
%Apparently he also used first the term subordination which is why the process is today usually refered to as 
The corresponding method is nowadays called subordination in the sense of Bochner, see 
\cite[Chapter 13]{schilling2012} for a thorough treatment. 

In terms of operator theory, this can be used to obtain fractional powers of an operator $A$ when $-A$ is the generator of a bounded $C_0$-semigroup.
The idea was subsequently extended in \cite{balakrishnan1959, hille1948, phillips1952} to what now is called the Hille-Phillips calculus.  

A.~V.~Balakrishnan further extended the construction of fractional powers in 1960 in \cite{balakrishnan1960} to the wider class of non-negative operators, 
a term coined by Komatsu in \cite{komatsu1969}. 
In the context of Banach spaces these are the linear operators having $(-\infty,0)$  contained in their resolvent set and fulfilling an additional resolvent estimate. 
In Banach spaces these operators coincide with sectorial operators as introduced in \cite{haase2006, kato1960} and differs slightly from the ones used for example in 
\cite{EngelNagel2000} or \cite{kato1980}. 
Especially generators of bounded $C_0$-semigroups are non-negative (and sectorial respectively). 
Further, in more general locally convex spaces non-negative operators are a strictly greater class (cf. \cite{martinez2001}). 

From a today's point of view Balakrishnan's construction is part of the modern calculus of sectorial operators, see e.g.\ \cite{delaubenfels1993, mcintosh1986}. 
A very detailed study of the entire topic is available in the great book \cite{haase2006}. 

In 1968 ideas on how to describe fractional powers of the Laplacian via extensions appeared in the context of stochastic processes (see \cite{molchanow1968}) but in this 
work focus was not on the fractional powers themselves. 
Namely the authors studied for $\alpha \in (0,1)$ the one-dimensional process generated by the operator
\[
 L_{\alpha} = \frac{\d^2}{\d t^2} + \frac{1-2\alpha}{t} \frac{\d}{\d t}. 
\]
Formally one can think of this process as the magnitude of a Brownian motion in $2 - 2 \alpha$-dimensions, since for an $n$-dimensional Brownian motion $(B_t)$ the process $(J_t)$ given by
$J_t := \norm{B_t}$, called a Bessel process of order $n$, is generated by the operator
\[
 L_{\frac{2-n}{2}} = \frac{\d^2}{\d t^2} + \frac{n - 1}{t} \frac{\d}{\d t}
\]
(sometimes with an additional factor $1/2$). 
Using the local time of the Bessel process the authors constructed a time change for an independent Brownian motion which yields the subordinated processes. 
The fact that this observation is actually independent of the special example of a Brownian motion seems to be noticed for the first time in \cite{nelson1958}. 

The approach was rediscovered 40 years later from the PDE point of view in the influential paper of Caffarelli and Silvestre \cite{caffarelli2007} where the authors described fractional powers of the Laplacian by means 
of taking traces of functions solving the PDE
\begin{equation}
\begin{aligned}
 \partial_t^2 u(t,x) + \frac{1-2\alpha}{t} \partial_t u(t,x) & =  - \Delta_x u (t,x), \quad & (t,x) \in (0, \infty) \times \R^n, \label{fractional_laplace} \\
 u(0,x) & = f(x), \quad & x \in \R^n,
\end{aligned} 
\end{equation}
with $\alpha \in (0,1)$ being the fractional power.
Formally one could interpret solutions to  \eqref{fractional_laplace} as harmonic functions defined on $\R^{n}\times\R^{2-2\alpha}$. 
In this case the equation \eqref{fractional_laplace} is nothing but the usual Laplacian applied to a function $v$ of the special form 
\begin{align*}
  v: \R^{n}\times\R^{2-2\alpha} \rightarrow \R, \quad  v(x,y) = u \left(\norm{y},x \right),  
\end{align*}
with a suitable function $ u:\R\times \R^{n} \rightarrow \R$. 
So it just depends on the norm of the additional $2-2\alpha$ coordinates. 
This may be the reason why the technique is sometimes also referred to as harmonic extension approach.
The here sketched idea is used nicely in \cite{caffarelli2007} to obtain results analogously to the ones available for the Laplacian in $\R^n$. 

Using a solution $u$ to \eqref{fractional_laplace}, one can then calculate $\left( - \Delta_x \right)^{\alpha}$ as
\[
 c_{\alpha} \bigl( ( - \Delta_x )^{\alpha}f \bigr)(x) = - \lim\limits_{t \rightarrow 0+} t^{1-2\alpha} \partial_t u(t,x), \qquad x\in\R^n,
\]
with an explicitly known constant $c_{\alpha}$ which was calculated for the first time in \cite{stinga2010} and a solution $u$ of \eqref{fractional_laplace}.  

One may generalise Equation \eqref{fractional_laplace} by replacing $-\Delta$ with a more general sectorial operator $A$ in some Banach space $X$. 
Now \eqref{fractional_laplace} becomes 
\begin{align}
 & u''(t) + \frac{1 - 2 \alpha}{t} u'(t) = Au(t), \quad t \in (0, \infty), \label{fractional_ODE} \\
 & u(0) = x, \label{dirichlet}
\end{align}
i.e.\ a linear ODE in the Banach space $X$ with initial datum $x \in X$ which degenerates for $t=0$ (unless $\alpha = 1/2$) and which is incomplete since no initial condition 
for $u'$ is given.
The case $\alpha = 1/2$ was already studied in \cite{balakrishnan1960} with the result that under the additional assumption of $u$ being bounded the problem has the 
unique solution given by
\[
 u(t) = \exp^{-t \sqrt{A}}x,
\]
i.e.\ is given by the analytic semigroup generated by $-\sqrt{A}$ (for more details including precise definitions the reader may consult \cite{haase2006} and 
\cite{martinez2001} or see below for the short introduction provided later in this paper). 
 
It was then noticed in \cite{stinga2010} that if the considered Banach space $X$ is one-dimensional, Equation \eqref{fractional_ODE} is another form of the Bessel differential equation and integral representations of its 
solutions provide solutions to \eqref{fractional_ODE} by interpreting them on the operator level. 
The underlying theory is the sectorial functional calculus (\cite{haase2006, isem21}) and the fact that the appearing functions can be approached within this framework. 

Assuming $u$ to be a bounded solution to \eqref{fractional_ODE} in the one-dimensional case it follows that $u$ is unique and that 
\begin{equation}
 - \lim\limits_{t \rightarrow 0+} t^{1-2\alpha} u'(t) = c_{\alpha} A^{\alpha}x.   \label{limit}
\end{equation}

Interpreting the limit on the one hand as generalised Neumann boundary condition ($u(0)=x$ is to be seen as Dirichlet boundary condition) we may think of the limit as generalised Dirichlet-to-Neumann operator which is given by
the fractional power of the sectorial operator appearing in the ODE. 
The question arises whether the solution $u$ is also unique in a general Banach space (that would mean that the Dirichlet-to-Neumann operator can be defined) and if this operator is up to a constant still given by a fractional 
power. 

In \cite{stinga2010} Stinga and Torrea studied this question in a Hilbert space $X$ for a sectorial selfadjoint operator $A$. 
Denoting the Dirichlet-to-Neumann operator (the limit in \eqref{limit}) by $T_{\alpha}$ they showed $T_{\alpha} = c_{\alpha} A^{\alpha}$ (without assuming $A$ to be injective) 
and the solution $u$ they used to construct $T_{\alpha}$ is unique if $A$ has pure point spectrum. 

The result was extended by Stinga in joint work with Gal{\'e} and Miana in \cite{gale2013} to general Banach spaces and actually also even to a bigger
class of operators. 
Namely in this work the authors considered generators of $\beta$-times integrated tempered semigroups and cosine families. 
Also it was noted that the ODE can be considered on an entire sector rather than just the positive halfline. 
With quite involved techniques the authors constructed solutions to the ODE and showed correspondence of the Dirichlet-to-Neumann operator constructed from the aforementioned
solution with a fractional power of the generator on the domain of the generator. 
Here it was noticed that the limit \eqref{limit} may be considered on an entire subsector of the domain of the solution. 
Also the paper allowed for complex powers being trapped in a strip. 
A complete uniqueness result for the solution is missing though. 

In \cite{arendt2016} the authors considered again the situation when $X$ is a Hilbert space and made use of form techniques to study the situation. 
In particular, they proved the well-posedness of the Dirichlet problem \eqref{fractional_ODE} for initial data as in \eqref{dirichlet} and showed that the domain of the 
Dirichlet-to-Neumann operator is a subspace of a complex interpolation space between $X$ and a dense subspace $V$ of it which contains $\mathcal{D}(A)$. 
The considered operator $A$ is m-accretive, i.e.\ sectorial with $M=1$, see Definition \ref{sectops} for more details.  
Such operators have bounded imaginary powers (\cite[Corollary 7.1.8]{haase2006}). 
By \cite[Theorem 11.5.4]{martinez2001} the domains of the fractional powers of these operators coincide with the complex interpolation spaces between $\mathcal{D}(A)$ and $X$ 
for real powers $\alpha$.
The question whether these interpolation spaces coincide with the domains of the Dirichlet-to-Neumann operator in the here treated sense was not completely clarified. 
Also the uniqueness result just holds in the coercive case.  

Our contribution is as follows. 
For a given sectorial densely defined operator $A$ on some general Banach space $X$ we construct a solution to the ODE \eqref{fractional_ODE} fulfilling the initial condition 
\eqref{dirichlet} which turns out to be holomorphic in some sector determined by the angle of sectoriality of $A$ and which is bounded on all closed subsectors 
(Theorem \ref{thm:existence_of_cesp_solution}).
This solution is the unique solution to the initial value problem by Theorem \ref{thm:uniqueness_of_cesp_solution}. 
If one uses it to define a Dirichlet-to-Neumann operator then this operator is up to the already known constant $c_{\alpha}$ given by the fractional power $A^{\alpha}$ 
(Theorem \ref{thm:correspondence_DtN_frac_power}). 
In a nutshell, we make use of the fact that $\sqrt{A}$ is sectorial with angle of sectoriality less than $\frac{\pi}{2}$ and thus gives rise to a holomorphic $C_0$-semigroup as well as a rich functional calculus.

The paper is organised in the following way. 
In the next section we will give a rather short introduction to the basics of sectorial operators on Banach spaces and the related sectorial functional calculus. 
The third section introduces holomorphic functions derived from modified Bessel functions which will turn out to be in the domain of the sectorial calculus of $A$ and which 
will be used for all results. 
The sections four and five deal with existence and uniqueness of bounded solutions to \eqref{fractional_ODE}, \eqref{dirichlet}.
We then relate the corresponding Dirichlet-to-Neumann operator of this solution with fractional powers of the operator $A$, and finish the paper with a short example.

\section{Sectorial Operators and Sectorial functional calculus}

Let $X$ be a Banach space over the field of complex numbers.

We denote by $L(X)$ the set of all bounded linear operators from $X$ to $X$ and 
\[
\rho(A) := \left\{ \lambda \in \C \mid (\lambda - A) \text{ is bijective} \right\}
\]
is the resolvent set of a closed linear operator $A$ in $X$, while $\sigma(A) := \C \setminus \rho(A)$ denotes its spectrum. 
\begin{defi}[Sectorial operator]
Let $A$ be a linear operator in $X$. 
Then $A$ is called \emph{sectorial} if $(- \infty, 0) \subseteq \rho\left(A\right)$ and
\[
 M := \sup\limits_{\lambda > 0} \norm{ \lambda \left( \lambda + A \right)^{-1}} < \infty.
\]
\label{sectops}
\end{defi}
In general we neither require $0 \in \rho(A)$ nor $\overline{\mathcal{D}(A)} = X$. 
Note that $M \in [1, \infty)$, see e.g.\ \cite[Corollary 1.1.4]{martinez2001}.
 
For $z \in \C \setminus (-\infty, 0]$ we define $\arg z$ to be the unique number in the interval 
$(-\pi, \pi)$ such that $z = \abs{z} \exp^{i \arg z}$. 
For $\theta \in (0, \pi)$ we denote by 
\[
	S_{\theta}:= \left\{ z \in \C \setminus (-\infty, 0]  
	                     \mid \abs{\arg z } < \theta \right\} 
\]
the open sector with vertex $0$ and opening angle $\theta$. 
Further we set $S_0 := (0, \infty)$. 

Let $A$ be a sectorial operator. Then there exists $\theta \in [0, \pi)$ such that $\sigma(A) \subseteq \overline{S_{\theta}}$ and for all 
$\phi \in (\theta, \pi)$ there exists $\widetilde{M}_{\phi} \geq 1$ such that $\sup_{z \in \C \setminus S_\phi} \norm{z(z-A)^{-1}} \leq \widetilde{M}_{\phi}$, 
see e.g.\ \cite[Proposition 1.2.1]{martinez2001}. 
The infimum of all angles $\theta$ such that $\sigma(A) \subseteq \overline{S_{\theta}}$ holds is called its 
\emph{angle of sectoriality} and denoted by $\omega$. 
By $\Sec _{\omega}(X)$ we denote the set of all sectorial operators with angle of sectoriality $\omega$.
Moreover, also $A+s$ is sectorial for all $s\geq 0$, and boundedly invertible for $s>0$.
Also, if $A$ is injective its inverse $A^{-1}$ is again a sectorial operator.  
%Operators such that $A + s$ is sectorial for some $s \in \R$ are called \emph{quasi-sectorial}. 

If $-A$ is the generator of a bounded $C_0$-semigroup $(\exp^{-tA})$ then $A$ is a sectorial operator with angle of sectoriality $\omega \leq \pi /2$. 
This follows from Laplace transform. 
Such semigroups are analytic on $S_{\pi/2-\omega}$ if and only if the angle of sectoriality $\omega$ is strictly less than $\pi /2$.
% In the following, $\left( \exp^{-zA} \right)_z$, will always 
% denote the semigroup generated by the sectorial operator $A$ (actually by its negative $-A$).   

Sectorial operators admit a powerful functional calculus, i.e.\ we can construct new linear operators by taking suitable holomorphic functions of a given sectorial operator.
%plugging in a given operator and a broad class of holomorphic functions. 
One may consult \cite{haase2006} and \cite{isem21} for an introduction and terminology. 

Let $G \subseteq \C$ be an open but possibly unbounded subset of the complex plane. 
By $\Hol(G)$ we shall denote the set of all holomorphic functions on $G$. 
We may turn them into a Fr\'echet space by equipping them with the topology of uniform convergence on compacts. 
For the calculus especially functions defined on a sector arbitrarily bigger than the one containing the spectrum of the considered operator $A$ will be important. 
Therefore we consider
\[
	\Hol [S_{\omega}] := \bigcup_{\varphi > \omega} \Hol (S_{\varphi})
\] 
which may be interpreted as the inductive limit of the Fr\'echet spaces $\Hol (S_{\varphi})$ for $\varphi>\omega$. 
By $\Hol^{\infty} (G)$ we denote the set of bounded holomorphic functions equipped with the norm topology.
This is a Banach space which continuously embeds in $\Hol (G)$. 
Similar to above we shall consider $\Hol^{\infty}[S_{\omega}]$, the inductive limit of the Banach spaces $\Hol^{\infty}(S_{\varphi})$ for $\varphi>\omega$.  
The set of \emph{elementary functions} $\mathcal{E} [S_{\omega}]$ consists of all $f \in \Hol^{\infty}[S_{\omega}]$ such that
\begin{equation}
	\exists \varphi > \omega \, \forall \, 0 \leq \alpha < \varphi: \, \int\limits_{0}^{\infty} \abs{f \left(s \exp^{i\alpha} \right)} \frac{\d s}{s} < \infty. 
	\label{int_cond}
\end{equation}
For $\omega \in [0, \pi)$, $A \in \Sec_{\omega}(X)$ and $f \in \mathcal{E} [S_{\omega}]$ one can define $f(A) \in L(X)$ via
\begin{equation}
	f(A) := \frac{1}{2 \pi i} \int_{\gamma} f(\lambda) (\lambda - A)^{-1} \d \lambda,
	\label{def:elem_fcts}
\end{equation}
with $\gamma$ given by the boundary of a sector of angle $\varphi>\omega$ such that $f$ is defined on the path $\gamma$. 
\begin{figure}[htb]
    \centering
    \begin{tikzpicture}[>=stealth]
      \draw[->] (-0.5,0.0)--(4.0,0.0) node[right]{$\mathrm{Re}$};
      \draw[->] (0,-2.0)--(0.0,2.0) node[above]{$\mathrm{Im}$};
      \draw (3,1.8)--(0.0,0.0);
      \draw[->] (3.0,1.8)--(1.5,0.9);
      \draw (0.0,0.0)--(3.0,-1.8);
      \draw[->] (0.0,0.0)--(1.5,-0.9);
      \draw (3.0,-1.3) arc(-90:-270:1.3);
      \draw (3.0,0.35) node{$\sigma(A)$};
      \draw (1.166190379,0.0) arc(0:30.96375653:1.166190379);
      \draw (0.8,-0.05) node[above]{$\phi$};
      \draw (3.0,-1.8) node[right]{$\gamma$};
    \end{tikzpicture}  
  \end{figure}
A standard argument shows that the definition is independent of the particular choice of the angle $\varphi$. 
A sufficient criterion for $f \in \mathcal{E}[S_{\omega}]$ is that $f \in \Hol^{\infty}[S_{\omega}]$ and that $f$ has polynomial limit $0$ at the origin and at infinity, i.e.
\[
	\exists C, \, \alpha> 0: \, \abs{f(z)} \leq C \cdot \min \{ \abs{z}^{\alpha}, \, \abs{z}^{-\alpha} \} \text{ on some sector } S_{\varphi}, \, \varphi > \omega.  
\]
So $f$ tends to $0$ at the origin and at infinity with a polynomial rate (see also \cite[p. 27 and Lemma 2.2.2.]{haase2006}). 
For such functions the integrability condition \eqref{int_cond} is satisfied. 

If the operator $A$ is not injective one needs to extend the algebra $\mathcal{E}[S_{\omega}]$ to
\[
	\mathcal{E}^{\text{ext}}[S_{\omega}] := \mathcal{E}[S_{\omega}] \oplus \operatorname{Lin}\bigl( z \mapsto (1+z)^{-1} \bigr) \oplus 
	\operatorname{Lin} \bigl( z \mapsto 1 \bigr).  
\]
For $f \in \mathcal{E}^{\text{ext}}[S_{\omega}]$ one has $f(z) = g(z) + c(1+z)^{-1} + d$ with $g \in \mathcal{E}[S_{\omega}]$, $c,d \in \C$ and we extend 
Definition \eqref{def:elem_fcts} to
\[
 f(A) := g(A) + c(1+A)^{-1} + d \in L(X). 
\]
Similarly to the case of $\mathcal{E}[S_{\omega}]$ a sufficient condition for $f \in \Hol^{\infty}[S_{\omega}]$ to be in $\mathcal{E}^{\text{ext}}[S_{\omega}]$ is the existence of limits at $0$ and at infinity which are approached
at a certain polynomial rate. 

A subset $\mathcal{N} \subseteq \mathcal{E}^{\text{ext}}[S_{\omega}]$ is said to be an \emph{anchor set} (\cite[p. 115]{isem21}) if
\[
 \bigcap\limits_{e \in \mathcal{N}} N(e(A)) = \{0\}
\]
where $N(B)$ denotes the null space of an operator $B$. 
If it happens that $\mathcal{N} = \{e\}$ for some single $e \in \mathcal{E}^{\text{ext}}[S_{\omega}]$ we call this $e$ an \emph{anchor element}. 
Anchor sets allow for a further extension of our calculus towards closed but in general unbounded operators. 
For this purpose we denote by $\Mer[S_\omega]$ the `germs' of meromorphic functions on the sector $S_\omega$. 
For $f \in \Mer [S_{\omega}]$ we consider 
$\mathcal{N}_f := \{ e \in \mathcal{E}^{\text{ext}}[S_{\omega}] \mid ef \in \mathcal{E}^{\text{ext}}[S_{\omega}]\}$ which is called the set of regularizers of $f$.  
If $\mathcal{N}_f$ is an anchor set we may define $f(A)$ by means of
\[
 x \in \mathcal{D}\bigl( f(A) \bigr)  \quad \textrm{and} \quad f(A)x := y \quad :\Leftrightarrow \quad \exists y \in X \, \forall e \in \mathcal{N}_f: \, e(A)y = (ef)(A)x
\]
and the fact that this is well-defined follows from the property that $\mathcal{N}_f$ is an anchor set. 
Especially if $f$ has an injective regularizer $e$ this just means
\[
 f(A) = e(A)^{-1} (ef)(A). 
\]
In this situation $\mathcal{D}\bigl( e(A)^{-1} \bigr) \subseteq \mathcal{D}\bigl(f(A)\bigr)$ by general functional calculus principles (see the Definition \cite[p. 31]{isem21} and \cite[Theorem 7.5]{isem21}). Let $\Mer_A:=\bigl\{f\in \Mer[S_\omega] \mid \mathcal{N}_f\,\text{is an anchor set}\bigr\}$. 
Note that for $A,B \in \Sec_{\omega}(X)$ and $f \in \mathcal{E}^{\textrm{ext}}[S_{\omega}]$, the bounded operators $f(A)$ and $f(B)$ can be defined 
(it just depends on the angle of sectoriality $\omega$) while $f \in \Mer_A$ does not necessarily imply $f \in \Mer_B$ and vice versa.

% At this point we could discuss further extensions of the calculus (\cite[p. 173]{isem21}) but by now we already have enough to suit the purposes of what comes which is why we 
% would stop at this point. 

Before we come to examples we want to state the so-called composition rule:

\begin{thm}[Composition rule, see {\cite[Theorem 2.4.2]{haase2006}}]
 Let $\omega \in [0, \pi)$, $A \in \Sec_{\omega}(X)$, $g \in \Mer_A$ such that $g(A) \in \Sec_{\omega'}(X)$ for some $\omega' \in [0, \pi)$ and for every $\varphi' \in (\omega', \pi)$ there is $\varphi \in (\omega, \pi)$ with
 the property that $g(S_{\varphi}) \subseteq \overline{S_{\varphi'}}$. Then
 \[
  \forall f \in \Mer_{g(A)}: \, f \circ g \in \Mer_A \text{ and } (f \circ g)(A) = f (g(A)). 
 \]
 \label{thm:comp_rule}
\end{thm}

\begin{example}
%We continue with two examples for the constructed calculus which are of major importance. 
Let $\omega \in [0, \pi/2)$ and $A\in\Sec_{\omega}(X)$. For $z\in S_{\pi/2 - \omega}$ choose $\varphi\in (\omega,\pi/2)$ such that $z\in S_{\pi/2 - \varphi}$ and consider the function 
$[S_\omega\ni \lambda\mapsto \exp^{-z\lambda}] \in \Mer _A$.
It holds that
\[
 \left[\lambda \mapsto \exp^{-z \lambda} \right](A) = \exp^{-zA}
\] 
where on the right hand side we have an operator of the holomorphic semigroup operator generated by $-A$. 

Since $zA$ is still sectorial we further can consider its functional calculus and it holds that
 \[
  \left[\lambda \mapsto \exp^{-\lambda} \right](zA) = \left[\lambda \mapsto \exp^{-z\lambda} \right](A)
 \]
 which is an instance of the composition rule. 
\end{example}

\begin{example}
Let $\omega \in [0, \pi)$ and $A \in \Sec_{\omega}(X)$. Then for all $\alpha \in \C_{\Re > 0}$ define
\[
	\C \setminus (-\infty,0] \ni z \mapsto z^{\alpha} := \exp^{\alpha \cdot \ln(z)}
\] 
with $[z \mapsto \ln z]$ being the inverse of the biholomorphic function $[z \mapsto \exp^z]$ restricted to the strip 
$\{ z \in \C \mid \abs{\Im (z)} < \pi \}$. 
Note that the function $[z\mapsto z^{\alpha}]$ allows for a continuous extension to
$z=0$ by setting $0^{\alpha} := 0$.  
We have $[\lambda\mapsto \lambda^{\alpha}] \in \Mer_A$ and set 
\[
 A^{\alpha} := \left[\lambda \mapsto \lambda^{\alpha} \right](A). 
\]
If $A$ is further injective this can be generalised to arbitrary $\alpha \in \C$ and it holds that
\[
 A^{-\alpha} = \left(A^{-1}\right)^{\alpha} = \left( A^{\alpha} \right)^{-1}. 
\]
%We will turn in a second to the special case $\alpha \in [0, \pi / \omega)$ but before we would like to mention the composition rule. 
\end{example}

The composition rule can be combined with the fact that for $A \in \Sec_{\omega}(X)$ and $\alpha \in [0, \pi/\omega)$ we have $A^{\alpha} \in \Sec_{\alpha \omega}(X)$ (see \cite[Proposition 3.1.2]{haase2006}). Therefore, 
\[
 \left[ \lambda \mapsto \exp^{-z\lambda} \right] \in \Mer_{\sqrt{A}} \quad \left( z \in S_{\frac{\pi}{2} - \varphi}, \frac{\omega}{2} < \varphi < \frac{\pi}{2} \right).
\] 
The semigroup $\left(\exp^{-z \sqrt{A}} \right)$ is holomorphic and strongly continuous on $\overline{\mathcal{D}(A)}$. 
It follows that $\mathcal{D}(A^{\infty}) := \bigcap_{k \in \N} \mathcal{D}(A^k)$ is dense in $\overline{\mathcal{D}(A)}$ (which is $X$ if we assume $A$ to be densely defined). 
We equip $\mathcal{D}(A^{\infty})$ with a family $(\norm{\cdot}_k)_{k\in\N}$ of seminorms given by $\norm{x}_k := \norm{A^k x}$ such that 
$\mathcal{D}(A^{\infty})$ becomes a Fr\'{e}chet space.

\begin{lemma}
 Let $\omega \in [0, \pi)$ and $A \in \Sec_{\omega}(X)$. 
 Further let $s \geq 0$ and $\delta \in (0, (\pi - \omega)/2)$. 
 Then there exists $\kappa > 0$ such that for all $k\in\N_0$ there exists $C > 0$ with the property that
  \begin{align*}
   \forall z \in S_{\delta}: \norm{ \left( \sqrt{A + s} \right)^k \exp^{-z\sqrt{A+s}}} 
   \leq C \left( 1 + \abs{z}^{-k} \right) \left(1 + \kappa \sqrt{s} \right)^k \exp^{-\kappa \Re z \sqrt{s}}.  
  \end{align*}
\label{lemma:exp_bound}
\end{lemma}
	
\begin{proof}
	Since $A \in \Sec_{\omega}(X)$, we have $A + s\in\Sec_{\omega}(X)$ for every $s \geq 0$ and hence 
	$\sqrt{A+s} \in \Sec_{\omega/2} (X)$ by the composition rule. 
	Therefore, $- \sqrt{A+s}$ generates a bounded analytic semigroup in $S_{(\pi - \omega)/2}$. 
	Now let $\lambda \in s + S_{\omega}$ for some $s \geq 0$. 
	We distinguish two cases. 
	Assume first $\omega \in (\pi/2, \pi)$. 
	Then
	\begin{align*}
		\Re \left( \lambda^{\frac{1}{2}} \right) & = \abs{\lambda}^{\frac{1}{2}} \cos \left( \frac{\arg (\lambda)}{2} \right) 
		\geq \abs{\lambda}^{\frac{1}{2}} \cos \left( \frac{\omega}{2} \right) \\
		& \geq \left( s \cdot \sin(\omega) \right)^{\frac{1}{2}} \cos \left( \frac{\omega}{2} \right) 
		= s^{\frac{1}{2}} \sin(\omega)^{\frac{1}{2}} \cos \left( \frac{\omega}{2} \right). 
	\end{align*}
	In the considered case set $c := \sqrt{\sin(\omega)}\cos(\omega / 2)$.
	\par
	In case one has $\omega \in [0, \pi / 2]$ the estimates become
	\[
	 \abs{\lambda}^{\frac{1}{2}} \cos \left( \frac{\omega}{2} \right) \geq s^{\frac{1}{2}} \cos \left( \frac{\omega}{2} \right) 
	\] 
	with $c:= \cos \left( \omega / 2 \right)$ in this case. 
	In both cases one has $c > 0$. 
    Let $f: \C \setminus (-\infty,0) \to \C$ be defined by $f(z) := \sqrt{z}$. 	
	Since $s + S_{\omega} \subsetneq S_{\omega}$ we obtain $f \left( s + S_{\omega} \right) \subseteq S_{\omega/2}$. 
	This together with the above inequality implies  
	$f \left( s + S_{\omega} \right) \subseteq S_{\omega/2} \cap \C_{\Re > c \sqrt{s}}$. 
	Choose $\kappa := c/2$, $\phi := \arctan (2 \cdot \tan(\omega/2))$. 
	Then $f \left( s + S_{\omega} \right) \subseteq \kappa \sqrt{s} + S_{\phi}$. 
	\begin{figure}[htb]
    \centering
    \begin{tikzpicture}[>=stealth]
      \draw[->] (-0.5,0)--(5.2,0) node[right]{$\mathrm{Re}$};
      \draw[->] (0,-0.5)--(0,4.1) node[above]{$\mathrm{Im}$};
      \draw (0,0)--(4.5,3.75);
      \draw (0,0)--(4.5/7.5,-0.5);
      \draw (1.30170825798,0) arc(0:39.8054584:1.30170825798);
      \draw (0.9,0.35) node{$\frac{\omega}{2}$};
      \draw[fill=black] (1.6,0.0) circle (1pt) node[below]{$\kappa \sqrt{s}$};
      \draw (3.2, 4.0) -- (3.2,-0.5);
      \draw[fill=black] (3.2,0.0) circle (1pt) node[below right]{$c \sqrt{s}$};
      \draw (2.6,0) arc(0:59.0361:1);
      \draw (2.3,0.35) node{$\phi$};
      \draw (1.6,0.0) -- (1.6+4.66477*0.514498, 0.0+4.66477*0.857492);
      \draw[pattern=north east lines] plot [smooth,tension=1] coordinates {(5,3.5) (4.5,3.2) (3.5,2) (4.5,0.5) (5,0.2)};
      \draw (4.9, 1.5) node[right]{$\sigma \left( \sqrt{A+s} \right)$};
    \end{tikzpicture}  
  \end{figure}
	\\
	Applying the spectral mapping theorem \cite[Theorem 2.7.8]{haase2006} we conclude that
	$\sigma \left( \sqrt{A+s} \right) \subseteq \kappa \sqrt{s} + S_{\phi}$. 
	Using \cite[Prop. 2.1.1]{lunardi1995} (which also holds for $z \in S_{\delta}$) it follows that for every $l \in \N_0$ there is $M_l > 0$ such that
	\[
	 \norm{z^l \left( \sqrt{A+s} - \kappa \sqrt{s} \right)^l \exp^{-z \sqrt{A+s}}} \leq M_l \exp^{-\kappa \Re(z) \sqrt{s}}. 
	\]
	This inequality with the aid of the binomial theorem finally results in
	\begin{align*}
	 \norm{\sqrt{A+s}^k \exp^{-z\sqrt{A+s}}} & \leq \sum\limits_{l=0}^k \binom{k}{l} 
	 \norm{\left(\sqrt{A+s} - \kappa \sqrt{s} \right)^l \exp^{-z\sqrt{A+s}}} \left(\kappa \sqrt{s} \right)^{k-l} \\
	 & \leq C \left(1 + \abs{z}^{-k} \right) \left(1 + \kappa \sqrt{s} \right)^k \exp^{-\kappa \Re (z) \sqrt{s}}
	\end{align*}
	with $C:= \sum_{l=0}^k \binom{k}{l} M_l$. 
\end{proof}

\section{Important holomorphic functions}

%Occasionally in this work we shall identify a given function $f$ with $f(z)$, the image of some 
%arbitrary $z \in \mathcal{D}(f)$ in order to ease notion. 
%For example we will refer to the function $\exp^{z}$ and mean $z \mapsto \exp^z$. 
%Different letters will be used for the variables in our functions but especially when we aim to plug in
%sectorial operators $\lambda$ shall be used in order to emphasise such points. 

Two functions which will play an important role in what comes are the modified Bessel functions 
$I_{\alpha}$ for $\alpha \in \C$ and $K_{\alpha}$ for $\alpha\in\C$ with $\Re \alpha\notin \Z$, both defined on $\C \setminus (-\infty,0]$ by
\begin{align*}
	I_{\alpha}(z) & := \left( \frac{z}{2} \right)^{\alpha} \sum\limits_{k=0}^{\infty} 
	\frac{z^{2k}}{4^k \cdot k! \cdot \Gamma(\alpha+k+1)},\\
	K_{\alpha}(z) & := \frac{\pi}{2 \sin (\alpha \pi)} \big( I_{-\alpha}(z) - I_{\alpha}(z) \big),
\end{align*}
where $\Gamma$ is the Gamma-function.

\begin{defi}
	Let $\alpha \in \C_{-1< \Re < 1} \setminus i \R$, $\omega \in [0, \pi /2)$, $z \in S_{\pi /2 - \omega}$ and $\delta \in (0, \pi / 2 - \omega - \abs{\arg z})$. 
	Define the function $u_z\from S_{\omega + \delta} \to \C$ by
	\begin{equation}
		u_z(\lambda) := 
		\begin{cases} 
		\frac{1}{2^{\alpha-1} \Gamma (\alpha)} (\lambda z)^{\alpha} K_{\alpha} (\lambda z) & \quad , \, \Re \alpha >0, \\
		\frac{1}{2^{-\alpha-1} \Gamma (-\alpha)} (\lambda z)^{-\alpha} K_{-\alpha} (\lambda z) & \quad , \, \Re \alpha <0. 
		\end{cases}
		\label{eq:def_u}	
	\end{equation}
\end{defi}

Note that we are actually interested in $\alpha\in\C$ such that $\Re \alpha \in (0,1)$ but extending the definition to $\alpha$ such that $\Re \alpha \in (-1,0)$ will be 
useful.
Further, note that if a statement on $u_z(\lambda)$ is true for $\Re \alpha\in (0,1)$, it automatically holds also for the case $\Re\alpha\in(-1,0)$ if one replaces every
$\alpha$ appearing in a statement by $-\alpha$.             
In the following lemma we will prove some properties of $u_z$. 

\begin{lemma}
	Let $\alpha \in \C_{0< \Re < 1}$, $\omega \in [0, \pi /2)$, 
	$z \in S_{\pi /2 - \omega}$ and $\delta \in (0, \pi / 2 - \omega - \abs{\arg z})$. Then
	\begin{enumerate}[label=(\alph*)]\itemsep3pt
		\item\label{item:1}
 		 We have $u_z\in \Hol\bigl(S_{\omega+\delta}\bigr)$. In particular, $u_z \in \Hol[S_\omega]$.

		\item\label{item:2}
		 We have
		    \[
				u_z(\lambda) - 1 = \mathcal{O}(\lambda ^{2 \alpha}) \quad \text{as } 
				\lambda \to 0 \text{ in } S_{\omega + \delta},
			\]			
			i.e.\ $u_z$ has polynomial limit $1$ at $\lambda = 0$. 
			\label{2nd_point}			
		\item\label{item:3}
			For $\lambda\in S_{\omega}$ we have $[z \mapsto u_z(\lambda)] \in \Hol(S_{\pi / 2 - \omega})$	and	
			\[
             \lim\limits_{\substack{z \to 0, \\ z \in S_{\frac{\pi}{2} - \omega}}} u_z(\lambda) = 1.  
			\]
		
		\item\label{item:4}
		 	It holds that
		 	\[
		 	 \lim\limits_{\substack{z \to 0, \\ z \in S_{\frac{\pi}{2} - \omega}}}  - z ^{1-2\alpha} \frac{\d}{\d z} u_z(\lambda) = c_{\alpha} \lambda^{2\alpha}
		 	\]
		 	with $c_{\alpha} = \frac{\Gamma (1-\alpha)}{2^{2\alpha-1} \Gamma (\alpha)}$.
		 	\label{4th_point}
		 	
		\item\label{item:5}
		    For $z \in S_{\pi/2 - \omega}$ and $\lambda\in S_{\omega+\delta}$ we have
		      \begin{equation}
			      u_{z}(\lambda) = \frac{z^{2 \alpha}}{2 \cdot \Gamma (2 \alpha)} \int\limits_0^{\infty} s^{\alpha - \frac{1}{2}} 
			      \exp^{-z \sqrt{\lambda^2 + s}} \left(\lambda^2 + s \right)^{-\frac{1}{2}} \d s. 
			      \label{eq:integralrepresentation_u}
		      \end{equation} 
		\item\label{item:6}
		      For $z \in S_{\pi/2 - \omega}$ there exist $C_1,C_2>0$ such that for all $\lambda\in S_{\omega + \delta}$ we have
			\[
		      \abs{u_z(\lambda)} \leq C_1 \exp^{-C_2 \abs{\lambda}}. 
			\]
			\label{6th_point}
			If further, for some $0 < \varepsilon < \pi/2 - \omega$ and $R > 0$, one has $z \in S_{\varepsilon}$, $\abs{z} < R$ one can choose $\delta$ and $C_3$ uniformly 
			such that $\abs{u_z(\lambda)} \leq C_3$. 
		\item\label{item:7}
			We have $u_z \in \Hol^{\infty}[S_{\omega}]$.			
		\item\label{item:8}
			We have	$u_z \in \mathcal{E}^{\text{ext}}[S_\omega]$.
	\end{enumerate}
	\label{lemma:properties_of_u}
\end{lemma}

\begin{proof}
	(a)
			Note that $[\lambda\mapsto \lambda z] \in \Hol (S_{\omega + \delta})$, and its range is contained in $S_{\pi / 2} \subsetneq \C \setminus (-\infty, 0]$ on which 
			the function $z\mapsto z^{\alpha} K_{\alpha}(z)$ is holomorphic. 

	(b)		We have the representations
			\[
			 u_z(\lambda) = \frac{\Gamma(1-\alpha)}{2^{\alpha}} \sum\limits_{k=0}^{\infty} \frac{1}{4^k \cdot k!} 
			 \left( \frac{(\lambda z)^{2k}}{2^{-\alpha}\Gamma(-\alpha+k+1)} - \frac{(\lambda z)^{2k+2\alpha}}{2^{\alpha}\Gamma(\alpha+k+1)} \right)
			\]
			for $\Re \alpha > 0$ and 
			\[
			 u_z(\lambda) = \frac{\Gamma(1+\alpha)}{2^{-\alpha}} \sum\limits_{k=0}^{\infty} \frac{1}{4^k \cdot k!} 
			 \left( \frac{(\lambda z)^{2k}}{2^{\alpha}\Gamma(\alpha+k+1)} - \frac{(\lambda z)^{2k-2\alpha}}{2^{-\alpha}\Gamma(-\alpha+k+1)} \right)
			\]
			for $\Re \alpha < 0$. 
			Note that we used Euler's formula of complements $\frac{\pi}{\sin (\alpha \pi)} = \Gamma (\alpha)  \cdot \Gamma (1- \alpha)$ , see 
			\cite[\S 5.5.3]{nist2019}, for the prefactors.
			Hence, for $0 < \Re \alpha < 1$, we obtain
			\begin{equation}
			 u_z(\lambda) = 1 - \frac{\Gamma(1-\alpha)}{2^{2\alpha} \Gamma (1+\alpha)}(\lambda z)^{2\alpha} + \mathcal{O} \left( (\lambda z)^2 \right)
			 \label{eq:exp1}
			\end{equation}
			while $-1 < \Re \alpha < 0$ will result in
			\begin{equation}
			 u_z(\lambda) = 1 + \frac{\Gamma(1+\alpha)}{2^{-2\alpha} \Gamma (1-\alpha)}(\lambda z)^{-2\alpha} + \mathcal{O} \left( (\lambda z)^2 \right).
			 \label{eq:exp2} 
			\end{equation}
		
   (c)
            The holomorphic function $[z \mapsto \lambda z]$ has its range contained in $S_{\pi/2}$. 
            Now the same argument as in the proof of \ref{item:1} applies.             
            The statement about the limit also follows from the expansions \eqref{eq:exp1} and \eqref{eq:exp2}. 		
		
	(d)		Note that by \cite[\S 10.29.4]{nist2019} we have
			\[
				\frac{1}{z} \frac{\d}{\d z} \exp^{\alpha \pi i} z^{\alpha} K_{\alpha}(z) = \exp^{(\alpha-1)\pi i} z^{\alpha -1} K_{\alpha -1}(z).
			\]
			Let us label $u_z$ with the used parameter $\alpha$ for this point and write $u_{z, \alpha}$.  
			Then we get
			\[
             - z ^{1-2\alpha} \frac{\d}{\d z} u_{z,\alpha}(\lambda) = c_{\alpha} \lambda^{2\alpha} u_{z,\alpha-1}(\lambda) \rightarrow c_{\alpha} \lambda^{2\alpha}   			
			\]
			as $z \to 0$ in $S_{\pi/2 - \omega}$ by \ref{item:3}. 
			
	(e)		The integral representation follows from \cite[\S 10.32.8]{nist2019} and the substitution $s:= t^2 - 1$ together with the Legendre duplication formula 
		    \cite[\S 5.5.5]{nist2019} for the prefactors. 
		    The fact that it also works for $\lambda = 0$ follows from a direct calculation. 
		     
	(f)		Note that
			\begin{align*}
				\abs{z^{\alpha}} \leq \abs{z}^{\Re \alpha} \cdot \exp^{\abs{\Im \alpha} \pi}, \quad 
				\abs{z^{\alpha}} = \abs{z}^{\alpha} \quad (\alpha \in \R), \quad
				\abs{z^{\alpha}} = z^{\Re \alpha} \quad (z \in \R),  \\
				\abs{z+w} \geq \left( \abs{z} + \abs{w} \right) \cdot 
				\cos \left( \frac{\arg z + \arg w}{2} \right) \quad (\abs{\arg z} + \abs{\arg w} < \pi). 
			\end{align*}
			Hence, starting from the integral representation \eqref{eq:integralrepresentation_u} of $u_z$ we estimate
			\[
				\abs{u_z(\lambda)} \leq \frac{\abs{z}^{2 \Re \alpha} \cdot \exp^{\abs{\Im \alpha}\pi}}{2 \abs{\Gamma(2\alpha)}} 
				\int\limits_{0}^{\infty} s^{\Re \alpha - \frac{1}{2}} \exp^{-\abs{z} \sqrt{\abs{\lambda^2 + s}} 
				\cos (\omega + \delta + \arg z)} \frac{1}{\sqrt{\abs{\lambda^2 + s}}} \d s.  
			\]
			Now make use of $\abs{\lambda^2+s} \geq (\abs{\lambda}^2 + s) \cos (\omega + \delta)\geq s\cos(\omega+\delta)$ and note $\sqrt{x^2+y^2} \geq 2^{-1/2} (x+y)$ $(x,y \geq 0)$ to further estimate
			\[	
				\abs{u_z(\lambda)}\leq \frac{\abs{z}^{2 \Re \alpha} \cdot \exp^{\abs{\Im \alpha}\pi}}{\sqrt{2} \abs{\Gamma(2\alpha)} \sqrt{\cos(\omega + \delta)}} 
				\int\limits_{0}^{\infty} s^{\Re \alpha - 1} \exp^{- \frac{\abs{z}}{\sqrt{2}}  
				(\abs{\lambda} + \sqrt{s}) \cos (\omega + \delta + \arg z) \sqrt{\cos(\omega + \delta)}} \d s,
			\]
			which yields the first part of the claim. 
			For the second part choose $\delta \in (0, \pi /2 - \omega - \varepsilon)$ and substitute $t := \abs{z}^2 s$ to get
			\[	
				\abs{u_z(\lambda)}\leq \frac{\exp^{\abs{\Im \alpha}\pi}}{\sqrt{2} \abs{\Gamma(2\alpha)} \sqrt{\cos(\omega + \delta)}} 
				\int\limits_{0}^{\infty} s^{\Re \alpha - 1} \exp^{- \frac{\sqrt{t}}{\sqrt{2}}  
				\cos (\omega + \delta + \varepsilon) \sqrt{\cos(\omega + \delta)}} \d t.
			\]
		
	(g)		This is a consequence of \ref{item:1} and \ref{item:6}. 
		
	(h)		Consider the function $\lambda\mapsto u_z(\lambda) - (1+\lambda^2)^{-1}$. 
			By \ref{item:6} it has polynomial limit $0$ at infinity and by \ref{item:2} the same holds
			at the origin.
\end{proof}

We will now define two more families of functions which will play an important role later on.

\begin{defi}
	Let $\alpha \in \C_{0< \Re < 1}$, $\omega \in [0, \pi /2)$, $z \in S_{\pi /2 - \omega}$ and $\delta \in (0, \pi / 2 - \omega - \abs{\arg z})$.
	Define the functions $v_z$ and $w_z$ on $S_{\omega+\delta}$ by
	\begin{align}
	\begin{split}
		v_z(\lambda) & := 	\frac{\Gamma (1-\alpha)}{2^{\alpha}} (\lambda z)^{\alpha} I_{-\alpha} (\lambda z), \\
		w_z(\lambda)	 & := \frac{\Gamma (\alpha)}{2^{1 - \alpha} } (\lambda^{-1} z)^{\alpha} I_{\alpha} (\lambda z).
    \end{split}	
	\label{eq:def_vw}	
	\end{align}
	\label{def:vw}
\end{defi}

Similarly to Lemma \ref{lemma:properties_of_u} we find properties of $v_z$ and $w_z$.

\begin{lemma}
 Let $\alpha \in \C_{0< \Re < 1}$, $\omega \in [0, \pi /2)$, $z \in S_{\pi /2 - \omega}$ and $\delta \in (0, \pi / 2 - \omega - \abs{\arg z})$.
	\begin{enumerate}[label=(\alph*)]\itemsep3pt
		\item\label{lemma:properties_of_vw:item:1}
		  We have $v_z, w_z \in \Hol \left( S_{\omega + \delta} \right)$. In particular, $v_z,w_z\in\Hol[S_\omega]$.
		\item\label{lemma:properties_of_vw:item:2}
		  $v_z(\lambda) - 1 = \mathcal{O}(\lambda)$ and $w_z(\lambda) - \frac{z^{2\alpha}}{2 \alpha} = \mathcal{O}(\lambda)$ as $\lambda \to 0$ in $S_{\omega + \delta}$,
		  i.e.\ both functions have polynomial limits at $\lambda = 0$. 
		\item\label{lemma:properties_of_vw:item:3}
		    For $\lambda \in S_\omega$ we have $[z \to v_z(\lambda)], [z \to w_z(\lambda)] \in \Hol ( S_{\pi / 2 - \omega})$. Moreover, we have		
			\[
			 \lim\limits_{\substack{z \to 0, \\ z \in S_{\frac{\pi}{2} - \omega}}} v_z (\lambda) = 1, 			 
			 \lim\limits_{\substack{z \to 0, \\ z \in S_{\frac{\pi}{2} - \omega}}}  - z ^{1-2\alpha} \frac{\d}{\d z} v_z(\lambda) = 0
			\]
			and
			\[
			 \lim\limits_{\substack{z \to 0, \\ z \in S_{\frac{\pi}{2} - \omega}}} w_z(\lambda) = 0, 			 
			 \lim\limits_{\substack{z \to 0, \\ z \in S_{\frac{\pi}{2} - \omega}}}  z ^{1-2\alpha} \frac{\d}{\d z} w_z(\lambda) = 1.
			\] 
		\item\label{lemma:properties_of_vw:item:4}
			For $c \in \R$ we have
			\[
				\exp^{-\lambda c} v_z(\lambda) = \mathcal{O}\left( \lambda^{\alpha-\frac{1}{2}} \exp^{-\lambda (c - z)} \right) \text{ and } 
				\exp^{-\lambda c} w_z(\lambda) = \mathcal{O}\left( \lambda^{-\alpha-\frac{1}{2}} \exp^{-\lambda (c - z)} \right)
			\]
			as $\abs{\lambda} \to \infty \text{ in } S_{\omega + \delta}$. 
			In particular, $[\lambda \mapsto \exp^{-\lambda c} v_z(\lambda)], [\lambda \mapsto \exp^{-\lambda c} w_z(\lambda)] \in \mathcal{E}^{\text{ext}}[S_{\omega}]$ if
			$c > \Re z \cdot \bigl( 1 + \tan( \abs{\arg z}) \cdot \tan(\omega+\delta) \bigr)$.  
		
		\item\label{lemma:properties_of_vw:item:5}
			For $A \in \Sec _{\omega}(X)$ we have		
			\[
				v_z, \, w_z \in \Mer_A. 
			\]
	\end{enumerate}
	\label{lemma:properties_of_vw}
\end{lemma}

\begin{proof}
	(a)	
   By definition we have
   \begin{equation}
    v_z(\lambda) = \Gamma(1-\alpha) \sum\limits_{k=0}^{\infty} \frac{z^{2k}\lambda^{2k}}{4^k \cdot k! \cdot \Gamma(-\alpha+k+1)}
    \label{eq:exp3}
   \end{equation}
   and
   \begin{equation}
    w_z(\lambda) = \frac{\Gamma(\alpha)}{2} z^{2\alpha} \sum\limits_{k=0}^{\infty} \frac{z^{2k}\lambda^{2k}}{4^k \cdot k! \cdot \Gamma(\alpha+k+1)}. 
    \label{eq:exp4}
   \end{equation}
   Hence both functions are entire. 
   In particular, they are holomorphic on every sector. 
   
  (b)
   This follows directly from the representations given in the proof of \ref{lemma:properties_of_vw:item:1}. 
  
  (c)
   Exploiting the expansions \eqref{eq:exp3} and \eqref{eq:exp4} we find
   \[
    v_z(\lambda) = 1 + \mathcal{O}(z^2) \quad \text{and} \quad w_z(\lambda) = \frac{z^{2\alpha}}{2\alpha} + \mathcal{O}(z^{2\alpha+2})
   \]
   and the statement follows. 
   
  (d)
   This follows from the asymptotic behaviour of $I_{\alpha}$, namely (see \cite[\S 2.1(iii), \S 10.40(i)]{nist2019})
   \[
    \forall \varepsilon \in \left( 0, \frac{\pi}{2} \right) \, \exists C,R > 0\, \forall z \in S_{\frac{\pi}{2} - \varepsilon}, \abs{z} > R: \, 
    \abs{I_{\alpha}(z) - \frac{\exp^z}{\sqrt{2\pi z}}} \leq \frac{C}{\abs{z}},
   \]
   i.e.\ $I_{\alpha} - (2\pi z)^{-1/2} \exp^z= \mathcal{O}(z^{-1})$ as $\abs{z} \to \infty$ for a given fixed sector in the right halfplane. 
   So one has (with $C > 0$ just depending on $\alpha$ and $z$):
   \[
    \exp^{-\lambda c} v_z(\lambda) = C \lambda^{\alpha} \exp^{-\lambda c} I_{-\alpha}(\lambda z) \sim C \lambda^{\alpha} \exp^{-\lambda c} 
    \frac{\exp^{\lambda z}}{\sqrt{2 \pi \lambda z}} 
   \]
   and similar for $w_z$ where we just have to replace $\lambda^{\alpha}$ by $\lambda^{-\alpha}$. It remains to estimate
   \[
    \Re (z \lambda) \leq \Re z \Re \lambda + \abs{\Im z} \abs{\Im \lambda} \leq \Re z \Re \lambda \bigl(1+ \tan(\abs{\arg z})\cdot \tan(\omega+\delta)\bigr). 
   \]
   Therefore, for $c > \Re z \bigl(1+ \tan(\abs{\arg z})\cdot \tan(\omega+\delta)\bigr)$, 
   \[
    \abs{ \lambda^{\alpha} \exp^{-\lambda c} \frac{\exp^{\lambda z}}{\sqrt{2 \pi \lambda z}}} \leq \exp^{\abs{\Im \alpha}\pi} \abs{\lambda}^{\Re \alpha}
    \frac{\exp^{\Re \left( \lambda z - \lambda c \right)}}{\sqrt{2 \pi \abs{\lambda z}}}, 
   \]
   which vanishes at infinity at an exponential rate. 
  
  (e)
   We have $[\lambda\mapsto \exp^{-\lambda c}] \in \mathcal{E}^{\text{ext}}[S_{\omega}]$, $\exp^{-cA}$ is injective by \cite[Corollary 2.1.7]{lunardi1995} and for 
   $c > \Re z (1+\tan(\omega+\delta) \cdot \tan(\abs{\arg z}))$ it holds that $[\lambda\mapsto \exp^{-\lambda c} v_z(\lambda)] \in \mathcal{E}^{\text{ext}}[S_{\omega}]$ 
   by \ref{lemma:properties_of_vw:item:2} and \ref{lemma:properties_of_vw:item:4} and the fact that polynomial limits at $0$ and at infinity suffice for a bounded function to be in 
   $\mathcal{E}^{\text{ext}}[S_{\omega}]$.
\end{proof}

In general the operators $v_z(A)$ and $w_z(A)$ will be unbounded. 
The following lemma will help to understand their domains. 

\begin{lemma}
 Let $\alpha \in \C_{0< \Re < 1}$, $\omega \in [0, \pi /2)$, $s \in (0,\infty)$, $\delta \in (0, \pi / 2 - \omega)$ and $c > s$.
 Define $f_{1,s}: S_{\omega + \delta} \to \C$ and $f_{2,s}: S_{\omega + \delta} \to \C$ by 
 \begin{align*}
  f_{1,s}(\lambda) := v_s(\lambda) - \frac{\Gamma(1-\alpha)s^{\alpha 
  - \frac{1}{2}}}{2^{\alpha} \sqrt{2\pi}}\left( 1 - \exp^{-c\lambda} \right) \lambda^{\alpha - \frac{1}{2}} \exp^{s \lambda}, \\ 
  f_{2,s}(\lambda) := w_s(\lambda) - \frac{\Gamma(\alpha)s^{\alpha 
  - \frac{1}{2}}}{2^{1 - \alpha} \sqrt{2\pi}}\left( 1 - \exp^{-c\lambda} \right)^2 \lambda^{-\alpha - \frac{1}{2}} \exp^{s \lambda}. 
 \end{align*}
 Then $f_{1,s}, f_{2,s} \in \mathcal{E}^{\mathrm{ext}}[S_{\omega}]$. 
\label{lemma:prep_domain_v_and_w}
\end{lemma}

\begin{proof}
 We only prove the statement for $f_{1,s}$. 
 The proof for $f_{2,s}$ is similar. 
 By the asymptotics of $I_{-\alpha}$ we may choose $C,R > 0$ such that
 \[
  \forall z \in S_{\omega + \delta}, \abs{z} > R: \, 
  \abs{I_{-\alpha}(z) - \frac{\exp^z}{\sqrt{2\pi z}}} \leq \frac{C}{\abs{z}}. 
 \]
 It follows that for $\lambda \in S_{\omega + \delta}$, $\abs{\lambda} > R / s$ we get
 \begin{equation}
  \begin{split}
      & \, \abs{f_{1,s}}(\lambda) \\   
 \leq & \, \frac{\abs{\Gamma(1 - \alpha)}}{2^{\Re \alpha}} \left( s\abs{\lambda} \right)^{\Re \alpha} \exp^{\abs{\Im \alpha} \pi} 
 \left( \frac{C}{s \abs{\lambda}} + \frac{\exp^{-(c-s) \Re \lambda}}{\sqrt{2\pi \abs{\lambda} s}} \right). 
  \end{split}
  \label{eq:asymptotics_f1s}
 \end{equation}
 So $f_{1,s}$ has polynomial limit $0$ at infinity. 
 As for the origin we already know by Lemma \ref{lemma:properties_of_vw} that $v_s$ has polynomial limit $1$ while for the other term one finds
 \[
  \abs{\frac{\Gamma(1-\alpha)s^{\alpha - \frac{1}{2}}}{2^{\alpha} \sqrt{2\pi}} \left( 1 - \exp^{-c\lambda} \right) \lambda^{\alpha - \frac{1}{2}} \exp^{s \lambda}} \leq
  \frac{\abs{\Gamma(1-\alpha)}s^{\Re \alpha - \frac{1}{2}}}{2^{\Re \alpha} \sqrt{2\pi}} c \abs{\lambda}^{\Re \alpha + \frac{1}{2}} 
  \exp^{\abs{\Im \alpha} \pi} \exp^{s \Re \lambda}
 \]
 as $\lambda \to 0$ in $S_{\omega+\delta}$. 
 This finishes the proof. 
\end{proof}

\section{Existence of a bounded continuous solution}
\label{mainchapter1}

Throughout this section $A$ is assumed to be a general sectorial operator with angle of sectoriality $\omega\in [0,\pi)$ in the Banach space $X$ and $\alpha \in \C_{0 < \Re < 1}$.

We formally consider the initial value problem 
\begin{equation}
 \begin{split}
  u''(t) + \frac{1 - 2 \alpha}{t} u'(t) & = Au(t) \qquad(t>0) \\
  u(0) & = x
 \end{split} \label{ODE}
\end{equation}
with initial datum $x \in X$.

\begin{defi}
  A \emph{solution of \eqref{ODE}} is a function $u \in C_{\text{b}} \bigl([0, \infty); X \bigr)$ such that $u(0) = x$, $u(t) \in \mathcal{D}(A)$ for $t > 0$, $Au \in L^1_{\mathrm{loc}} \bigl( (0, \infty);X \bigr)$ and \eqref{ODE} 
  is satisfied in the sense of distributions, i.e.\ for all $\phi\in C^{\infty}_{\mathrm{c}} \bigl( (0,\infty) \bigr)$ we have
  \[
	  \int\limits_{(0,\infty)} \bigl( u(t) \phi ''(t) - u(t) \frac{\d}{\d t} \frac{1-2\alpha}{t} \phi (t) \bigr) 
	  \, \d t 
	  = \int\limits_{(0,\infty)} Au(t) \phi (t) \, \d t. 
  \]
\end{defi}

% A function $u \in C_{\text{b}} \bigl([0, \infty); X \bigr)$ is considered to be a \emph{solution} \textbf{Highlighting gewuenscht?} of \eqref{ODE} if $u(0) = x$, $u(t) \in \mathcal{D}(A)$ for $t > 0$, $Au \in L^1_{\text{loc}} \bigl( (0, \infty);X \bigr)$ and \eqref{ODE} 
% is satisfied in the sense of distributions, i.e.
% \[
% 	\forall \phi \in C^{\infty}_{\text{c}} \bigl( (0,\infty) \bigr): 
% 	\, 	\int\limits_{(0,\infty)} \bigl( u(t) \phi ''(t) - u(t) \frac{\d}{\d t} \frac{1-2\alpha}{t} \phi (t) \bigr) 
% 	\, \d t 
% 	= \int\limits_{(0,\infty)} Au(t) \phi (t) \, \d t. 
% \]
% We will refer to the entire problem as CSEP. 

\begin{remark}
\label{remark:better_regularity}
	The ODE \eqref{ODE} may be rewritten as
	\[
		\frac{\d}{\d t} t^{1-2\alpha} u'(t) = t^{1-2\alpha} Au(t). 
	\]
	Hence, for a solution $u$ one directly gets that $u \in C^1 \bigl( (0, \infty); X \bigr) \cap W^{2,1}_{\text{loc}} \bigl( (0, \infty); X \bigr)$. 
\end{remark}

Equation \eqref{ODE} is nothing but a slightly different version of the modified Bessel equation. 
Namely, given a solution $v$ to the modified Bessel equation 
\[ 
 t^2 v''(t) + t v'(t) - (t^2 + \alpha^2)v(t) = 0
\]
one obtains a solution to \eqref{ODE} by considering the function $u$ defined by 
$u(t) := \left( \sqrt{\lambda} t \right)^{\alpha} v\left( \sqrt{\lambda} t \right)$, where 
$\lambda \in \C \setminus (-\infty, 0]$ is a one-dimensional version of our sectorial operator $A$. 

\begin{remark}
It was first noticed in \cite{stinga2010} that this connection can be used to state an explicit formula
for the solution which was further generalised with respect to the space and the assumptions on the operator $A$
under consideration in \cite{gale2013}.
\end{remark}
%This can be done by using the following expression.

\begin{defi}
 For $z \in S_{(\pi-\omega)/2}$ we define  
 \begin{equation}
  U(z) := \frac{z^{2\alpha}}{2 \cdot \Gamma(2 \alpha)} \int\limits_0^\infty s^{\alpha - \frac{1}{2}} 
  \exp^{-z \sqrt{A + s}} (A + s)^{-\frac{1}{2}} \, \d s.
  \label{defi_U}
 \end{equation}
\end{defi}

\begin{lemma}
 The Bochner-integral in \eqref{defi_U} is convergent in $L(X)$. 
 Further $U$ is uniformly bounded on every proper subsector $S_{\delta} \subsetneq S_{(\pi-\omega)/2}$.  
 Finally let $k \in \N_0$. 
 Then $A^k U \in \Hol \bigl( S_{(\pi-\omega)/2}; L(X) \bigr)$. 
 \label{lemma:U_is_smooth}
\end{lemma}

\begin{proof}
 Using Lemma \ref{lemma:exp_bound} we can estimate for $z \in S_{\delta}$
 \[
  \int\limits_0^{\infty} \norm{s^{\alpha - \frac{1}{2}} \exp^{-z \sqrt{A + s}} 
  (A + s)^{-\frac{1}{2}}} \, \d s \leq C \cdot M \int\limits_0^{\infty} s^{\Re \alpha - 1} 
  \exp^{-\kappa \Re z \sqrt{s}} \d s < \infty, 
 \]
 for some constants $C > 0$ depending on $S_{\delta} \subsetneq S_{(\pi-\omega)/2}$, $\kappa > 0$ depending on $\omega$ and 
 $M \geq 1$ which is the non-negativity constant of the sectorial operator $\sqrt{A}$. 
 Hence, the integral converges in $L(X)$.  
 For the boundedness we estimate
 \begin{align*}
  \norm{U(z)} & \leq \frac{C \cdot M \cdot \abs{z}^{2 \Re \alpha} \exp^{\abs{\Im \alpha} \pi}}{2 \cdot \abs{\Gamma(2\alpha)}} \int\limits_0^{\infty} s^{\Re \alpha - 1} 
  \exp^{-\kappa \Re z \sqrt{s}} \d s \\ 
  & = \frac{C \cdot M \cdot \abs{z}^{2 \Re \alpha} \exp^{\abs{\Im \alpha}} \left( \Re z \right)^{-2 \Re \alpha}}{2 \cdot \abs{\Gamma(2\alpha)}} 
  \int\limits_0^{\infty} t^{\Re \alpha - 1} \exp^{-\kappa \sqrt{t}} \d t \\
  & \leq \frac{C \cdot M}{2 \abs{\Gamma(2\alpha)} \cos^{2 \Re \alpha}(\delta)} \int\limits_0^{\infty} t^{\Re \alpha - 1} \exp^{-\kappa \sqrt{t}} \d t
 \end{align*}
 where we used the substitution $t:= \left( \Re z \right)^2 s$ and the inequality $\Re z \geq \abs{z} \cos(\delta)$. 
 Further, for $n, k \in \N_0$ we calculate
 \begin{align*}
  \frac{\d ^ n}{\d z^n} A^k  \exp^{-z\sqrt{A+s}} & = (-1)^n A^k (A+s)^{\frac{n}{2}} \exp^{-z\sqrt{A+s}} \\ 
  & = (-1)^n \sum\limits_{l=0}^k \binom{k}{l} (-s)^{k-l} (A+s)^{l+\frac{n}{2}} \exp^{-z\sqrt{A+s}},
 \end{align*}
 which yields holomorphy of the integral by applying again Lemma \ref{lemma:exp_bound} to every single summand in the above sum and combining it with Hille's theorem and 
 differentiation under the integral sign. Concerning the prefactor $z^{2\alpha}$ one just has to note that it is smooth on $S_{(\pi-\omega)/2}$.  
\end{proof}

In general, unless $A$ is bounded, the function $U$ lacks continuity at $z=0$ in the norm-topology of $L(X)$.  
One has strong continuity though. In order to see this we start with a lemma which tells us that we already encountered the function $U$. 

\begin{lemma}
	Let $\omega \in [0, \pi)$ and $A \in \Sec _{\omega}(X)$. Then
	\[
		U(z) = u_z(\sqrt{A})
	\]
	for $z\in S_{(\pi - \omega)/2}$,
	where $u_z$ is the function defined in \eqref{eq:def_u}. 
\end{lemma}

\begin{proof}
	We have $\sqrt{A} \in \Sec _{\omega/2}(X)$. 
	Moreover, Lemma \ref{lemma:properties_of_u} states that we have $u_z \in \mathcal{E}^{\text{ext}}[S_{\omega/2}]$. 
	By the composition rule in Theorem \ref{thm:comp_rule} we have $[\lambda\mapsto u_z(\sqrt{\lambda})] \in \mathcal{E}^{\text{ext}}[S_\omega]$ and
	$\bigl[\lambda\mapsto u_z(\sqrt{\lambda}) - \tfrac{1}{1+\lambda}\bigr]\in\mathcal{E}[S_\omega]$ and analogously for $s > 0$ and $z \in S_{(\pi - \omega)/2}$ 
	 it holds that $[\lambda \mapsto \frac{\exp^{-z\sqrt{\lambda + s}}}{\sqrt{\lambda + s}}] \in \mathcal{E}^{\rm{ext}}[S_{\omega}]$, 
	$[\lambda \mapsto \frac{\exp^{-z\sqrt{\lambda + s}}}{\sqrt{\lambda + s}} - \frac{\exp^{-z\sqrt{s}}}{(1 + \lambda)\sqrt{s}}] 
	\in \mathcal{E} [S_{\omega}]$. 
	With $\gamma$ as the boundary of a suitable sector of angle $\varphi>\omega$ (see \eqref{def:elem_fcts}), we calculate
	\begin{align*}
		u_z(\sqrt{A}) & = \frac{1}{2\pi i} \int\limits_{\gamma} \left(u_z(\sqrt{\lambda}) - \frac{1}{1+\lambda} \right) (\lambda-A)^{-1}\d \lambda + (1+A)^{-1} \\
		              & = \frac{1}{2\pi i} \int\limits_{\gamma} \left( \frac{z^{2 \alpha}}{2 \Gamma (2 \alpha)} \int\limits_0^{\infty} s^{\alpha - \frac{1}{2}}
	                      \left( \frac{\exp^{-z \sqrt{\lambda + s}}}{\sqrt{\lambda + s }} - \frac{\exp^{-z \sqrt{s}}}{(1+\lambda)\sqrt{s}} \right) \d s \right) (\lambda-A)^{-1}\d \lambda \\
	                  & + (1+A)^{-1} \\
	                  & = \frac{z^{2 \alpha}}{2 \Gamma (2 \alpha)} \int\limits_0^{\infty} s^{\alpha - \frac{1}{2}} \left( \exp^{-z \sqrt{A + s}} (A + s)^{-1/2} -
	                      (1 + A)^{-1} \frac{\exp^{-z \sqrt{s}}}{\sqrt{s}} \right) \d s \\ 
	                  & + (1+A)^{-1} \\
	                  & = U(z). \qedhere
	\end{align*}
\end{proof}

We can now prove the claimed strong continuity. 
More precisely, we are going to show that $U$ is strongly continuous on $\overline{\mathcal{D}(A)}$.
%Our technique is the same as one uses when considering holomorphic semigroups in functional calculus. 

\begin{lemma}
	Let $\omega\in [0,\pi)$, $A \in \Sec _{\omega}(X)$, $\delta \in \bigl[0, (\pi-\omega)/2 \bigr)$ and $x\in X$. Then
	\[
		x \in \overline{\mathcal{D}(A)} \; 
		\Leftrightarrow \; \lim\limits_{\stackrel{z \to 0}{z \in S_{\delta}}} U(z)x=x. 
	\]
	\label{lemma:U_is_strongly_continuous}
\end{lemma}

\begin{proof}
     Let $U(z)x\to x$.
	By Lemma \ref{lemma:U_is_smooth} we have $U(z)x \in \mathcal{D}(A^{\infty})$ and therefore 
	$x \in \overline{\mathcal{D}(A)}$. 
	
	Conversely, for $z\in S_{\delta}$ we consider the function $f_z\from S_{\pi - 2\delta} \to \C$ defined by
	\[
		f_z(\lambda) := \frac{u_z(\sqrt{\lambda}) - 1}{(1+\lambda)} = 
		z^{\alpha} \frac{u_z{(\sqrt{\lambda} ) - 1}}{\left(z \lambda \right)^{\alpha}} 
		\frac{\lambda^{\alpha}}{1+\lambda}. 
	\]
	The expansion \eqref{eq:exp1} yields
	\[
	 \frac{u_z{(\sqrt{\lambda} ) - 1}}{\left(z \lambda \right)^{\alpha}} \to c_{\alpha} z^{\alpha}	
	\]
	as $\lambda \to 0$ in $S_{\pi - 2\delta}$. 	
	Hence this part is bounded on the entire sector $S_{\pi - 2\delta}$. 
	It follows that the net $(f_z)$, $z \in S_{\delta},\abs{z} < 1$ is in $\mathcal{E}[S_{\omega}]$ and for all $\varphi \in [0, \pi - 2\delta)$ one has
	\[
	 \int\limits_0^{\infty} \frac{ \abs{f_z \left( s \exp^{i \varphi}\right)} }{s} \d s  \leq C \abs{z}^{\Re \alpha} \int\limits_0^{\infty} 
	 \frac{ s^{\Re{\alpha}}}{s \abs{1 + s \exp^{i\varphi}}} \d s 
	\]
	where $C > 0$ is independent of $z$. Hence,
	\[
	 \lim\limits_{z \to 0} \int\limits_0^{\infty} \frac{ \abs{f_z \left( s \exp^{i \varphi}\right)} }{s} \d s = 0 
	\]
	and therefore, using that
	\[
	 \norm{f_z(A)} \leq \frac{M}{2\pi} \int\limits_0^{\infty} \frac{ \abs{f_z \left( s \exp^{i \varphi}\right)} }{s} \d s
	\]
	where $M \geq 1$ is the non-negativity constant of $A$, it follows that
	\[
		f_z(A) = \left( U(z) - 1 \right) \left( 1+A \right)^{-1} \to 0
	\]
	in $L(X)$ as $z \rightarrow 0$ in $S_{\delta}$ which shows
	\[
	 \lim\limits_{\stackrel{z \to 0}{z \in S_{\delta}}} U(z)x=x
	\] 
	for $x \in \mathcal{D}(A)$. 
	The final claim follows now from the fact that $U(z)$ is uniformly bounded on $S_{\delta}$ by Lemma \ref{lemma:U_is_smooth}, 
	so the limit holds for all $x\in\overline{\mathcal{D}(A)}$. 
\end{proof}

\begin{prop}
    Let $A\in \Sec_{\omega}(X) $.
	Then $U$ solves the ordinary differential equation
	\[
		U''(z) + \frac{1-2\alpha}{t}U'(z) = AU(z) \quad (z \in S_{\frac{\pi - \omega}{2}}),
	\]
	in $L(X)$. 
	\label{prop:U_is_solution}
\end{prop}

\begin{proof}
 By linearity we can forget about constant prefactors. For $z \in S_{(\pi - \omega)/2}$ define
 \[
  G(z) := z^{2 \alpha} \int\limits_0^{\infty} F(z,s) \d s \quad \text{with} \quad 
  F(z,s):= s^{\alpha - \frac{1}{2}} \exp^{-z \sqrt{A + s}} (A + s)^{-\frac{1}{2}}.  
 \]
 One calculates
 \[
  G'(z) = \frac{2\alpha}{z}G(z) - z^{2\alpha}\int\limits_0^{\infty} \sqrt{A+s} \cdot F(z,s) \d s
 \]
 and
 \[
 \begin{split}
  G''(z) =  - & \, \frac{2\alpha}{z^2}G(z) + \frac{2\alpha}{z} G'(z) + z^{2\alpha}\int\limits_0^{\infty} (A+s) \cdot F(z,s) \d s \\
            - \, & \frac{2 \alpha}{z} z^{2\alpha}\int\limits_0^{\infty} \sqrt{A+s} \cdot F(z,s) \d s. 
 \end{split}
 \]
 Thus, for $z \in S_{(\pi - \omega)/2}$ we obtain
 \[
 \begin{split}
  G''(z) + \frac{1-2\alpha}{z} G'(z) & = - z^{2\alpha} \int\limits_0^{\infty} \left( \frac{1-2\alpha}{z} s^{\alpha - \frac{1}{2}} \exp^{-z \sqrt{A + s}}\right) \d s + AG(z) \\
  & = - z^{2\alpha} \int\limits_0^{\infty} \frac{2}{z} \frac{\d}{\d s} \left( s^{\alpha + \frac{1}{2}} \exp^{-z \sqrt{A+s}} \right) \d s + AG(z) \\
  & = AG(z), 
 \end{split}
 \]
 where we used integration by parts and afterwards Hille's Theorem in the last two steps.    
\end{proof}

As a consequence we get the following theorem, which is the main result of this section.

\begin{thm}
	If $\mathcal{D}(A)$ is dense in $X$ the function $u(z) := U(z)x$ is a solution to \eqref{ODE}. 
	\label{thm:existence_of_cesp_solution}
\end{thm}

\begin{proof}
	By Proposition \ref{prop:U_is_solution}, $u$ solves the ODE since $U$ does so in $L(X)$. 
	Strong continuity follows from Lemma \ref{lemma:U_is_strongly_continuous} together with the assumption that $\overline{\mathcal{D}(A)} = X$. 
	The boundedness of $u$ is a consequence of Lemma \ref{lemma:U_is_smooth}. 
\end{proof}

\section{Uniqueness of the solution}

In this section we shall show the uniqueness of the solution to \eqref{ODE} found in Theorem \ref{thm:existence_of_cesp_solution}. 
In order to obtain uniqueness we will exploit the asymptotic behaviour of $I_{\alpha}$ on a fixed sector $S_{\delta}$. Remember, by 
\cite[\S 2.1(iii), \S 10.40(i)]{nist2019}, we have
\begin{equation}
	I_{\alpha}(z) \sim \frac{\exp^{z}}{\sqrt{2 \pi z}} \text{ as } \abs{z} \to \infty. 
	\label{eq:asymptotics}
\end{equation}
Let us explain the basic idea for showing uniqueness of the solution.
First, we uniquely solve a modified problem similar to the one under consideration but with initial conditions for $u$ and `$u'$' (actually a scaled version of it). 
The resulting unique solution will in general be unbounded unless a constraint couples the initial data leading to a Dirichlet-to-Neumann operator which will be studied in the next section. 

Again, let $X$ be a Banach space and $\alpha\in \C_{0<\Re<1}$.
As a first step, we show uniqueness of the solution of \eqref{ODE} for bounded sectorial operators. 
By approximation we will then generalise this to unbounded operators. 
In the last step the asymptotic behaviour of $I_{\alpha}$ together with the assumption of a bounded solution for \eqref{ODE} will give uniqueness as sketched above. 

Note that for a sectorial operator $A$ and $c \geq 0$ the operator $\exp^{-c\sqrt{A}}$ is injective. 
We set
\[
 \exp^{c\sqrt{A}} = \left( \exp^{-c\sqrt{A}} \right)^{-1}
\]
with $\mathcal{D}(\exp^{c\sqrt{A}}) = R(\exp^{-c\sqrt{A}})$. 

Throughout this section, we will make use of the functions $v_z$ and $w_z$ introduced in Definition \ref{def:vw}.

We start with some preparation.
\begin{lemma}
 Let $\omega \in [0, \pi/2)$, $A \in \Sec_{\omega}(X)$ and $0 < s < t$. 
 Then we have $\mathcal{D} \left( v_t(A) \right) \subseteq \mathcal{D} \left( v_s(A) \right)$ and 
 $\mathcal{D} \left( w_t(A) \right) \subseteq \mathcal{D} \left( w_s(A) \right)$. 
 \label{lemma:nested_domains}
\end{lemma}

\begin{proof}
 We will only proof the statement for $v_t(A)$ and $v_s(A)$, respectively. 
 The second case is similar. 
 Consider the function $f_{1,s}$ from Lemma \ref{lemma:prep_domain_v_and_w}.
 By this lemma, $f_{1,s}(A) \in L(X)$. 
 Let $x \in \mathcal{D} \bigl(v_s(A) \bigr)$ and choose $d > s$.
 Write $f_{1,s}(\lambda) = v_s(\lambda) + g_s(\lambda)$, i.e., 
 \[
  g_s(\lambda) := -\frac{\Gamma(1-\alpha)s^{\alpha - \frac{1}{2}}}{2^\alpha\sqrt{2\pi}} (1-\exp^{-c\lambda})\lambda^{\alpha - \frac{1}{2}}\exp^{s\lambda},
 \] 
 and observe that $g_s \in \Mer_A$. 
 Then $\exp^{-d A}x \in \mathcal{D}\bigl( g_s(A) \bigr)$ and
 \[
  \exp^{-d A} \bigl( v_s(A)x - f_{1,s}(A)x \bigr) = g_s(A) \exp^{-d A}x. 
 \]  
It follows that $x \in \mathcal{D}\bigl( g_s(A) \bigr)$ and by symmetry of the argument one obtains $\mathcal{D}\bigl( g_s(A) \bigr) = \mathcal{D} \bigl(v_s(A) \bigr)$. 
In particular, $\mathcal{D}\bigl( g_s(A) \bigr)$ is independent of the constant $c > s$ appearing in its definition (cf.\ the definition of $f_{1,s}$ in Lemma 
\ref{lemma:prep_domain_v_and_w}). 
Next observe that if we define $h: S_{\omega + \delta} \to \C$ by 
\[
 h(\lambda) := -\frac{\Gamma(1-\alpha)s^{\alpha - \frac{1}{2}}}{2^\alpha\sqrt{2\pi}}\left(1 - \exp^{-c\lambda} \right)\lambda^{\alpha - \frac{1}{2}},
\]
we have $h \in \Mer_A$ and by general principles
\[
 \exp^{sA} h(A) \subseteq g_s(A). 
\]
We claim that equality holds. 
In order to see this let $x \in \mathcal{D} \bigl( g_s(A) \bigr)$ be given. 
Further let $\varepsilon >0$. 
Then $\exp^{-\varepsilon A}(1+A)^{-1}x \in \mathcal{D} \bigl( h(A) \bigr)$ and 
\[
 h(A) \exp^{-\varepsilon A}(1+A)^{-1}x = \exp^{-\varepsilon A} \exp^{-s A} (1+A)^{-1} g(A)x. 
\]
Since the analytic semigroup $\left(\exp^{-\varepsilon A} \right)$ is strongly continuous on $\mathcal{D}(A)$ and $h(A)$ is closed in $X$ we can let $\varepsilon \to 0\rlim$
and afterwards may apply $(1 + A)$ to get
\[
 h(A)x = \exp^{-s A }g_s(A)x
\]
which yields the equality claim. 
Finally it follows that
\begin{align*}
 \mathcal{D} \bigl( v_s(A) \bigr) = \mathcal{D} \bigl( \exp^{sA} h(A) \bigr) & = 
 \left\{ x \in \mathcal{D} \bigl( h(A) \bigr) \mid h(A)x \in \mathcal{D} ( \exp^{sA} ) \right\} \\
 & \supseteq \left\{ x \in \mathcal{D} \bigl( h(A) \bigr) \mid h(A)x \in \mathcal{D} ( \exp^{tA} ) \right\} = \mathcal{D} \bigl( v_t(A) \bigr)
\end{align*}
for $s < t$ since in this situation $\mathcal{D} \bigl( \exp^{tA} \bigr) \subseteq \mathcal{D} \bigl( \exp^{sA} \bigr)$. 
This yields the claim of the lemma. 
\end{proof}

\begin{lemma}
 Let $\omega \in [0, \pi)$,  $A \in \Sec_{\omega}(X)$, $c>0$, $x,y \in \mathcal{D} ( \exp^{c \sqrt{A}} )$ and $\delta\in \bigl( 0, (\pi-\omega)/2 \bigr)$.
 Then it holds that 
 \[
  [z \mapsto v_z(\sqrt{A})x], \, [z \mapsto w_z(\sqrt{A})y] \in \Hol \Bigl( S_{\frac{\pi-\omega}{2} - \delta} \cap \C_{0< \Re < d} \, ; \, \mathcal{D}(A^{\infty}) \Bigr), 
 \]
 where 
 \[
  0 < d < \frac{c}{ 1 + \tan \left( \frac{\pi - \omega}{2} - \delta \right) \cdot \tan(\frac{\omega}{2}+\delta) }.
 \] 
 Furthermore, $u\from S_{(\pi-\omega)/2 - \delta}\to \mathcal{D}(A^{\infty})$ defined by
 \[
  u(z) := v_{z}(\sqrt{A})x + w_{z}(\sqrt{A})y,
 \]
 solves
 \begin{align*}
    u''(z) + \frac{1-2\alpha}{z}u'(z) & = Au(z) \quad (z \in S_{(\pi-\omega)/2 - \delta} \cap \C_{0< \Re < d}),\\
    \lim\limits_{z \to 0} u(z) & = x,\\
    \lim_{z\to 0\rlim} z^{1-2\alpha} u'(z) & = y. 
 \end{align*} 
 \label{lemma:solution_to_modified_problem}
\end{lemma}

\begin{proof}

 As abbreviation we set $G := S_{(\pi-\omega)/2 - \delta} \cap \C_{0< \Re < d}$.  
 Going back to the definition of the two functions $v_z$ and $w_z$ one sees that for fixed $n \in \N$, $z \in G$ the functions 
 $[\lambda \mapsto \partial_z^n v_z(\lambda)]$ and $[\lambda \mapsto \partial_z^n w_z(\lambda)]$ are still entire in $\lambda$, i.e., polynomial limits at $0$ still exist.  
 As for infinity note that a derivative of arbitrary order of a Bessel function may be expressed as a finite sum of Bessel functions (of the same kind) of varying order 
 (\cite[\S 10.29.5]{nist2019}), i.e.,
 \[
  \frac{\d^n }{\d z^n} \exp^{i \alpha \pi} I_{\alpha}(z) = \frac{1}{2^n} \sum\limits_{k=0}^n \binom{n}{k} \exp^{i(\alpha-n+2k)\pi} I_{\alpha-n+2k}(z). 
 \] 
% Also, the asymptotic behaviour of modified Bessel functions is almost independent of its order at least on the lowest level. 
% By this, we mean that 
% \[
%  \forall \alpha \in \C, \, \varepsilon \in \left( 0 , \frac{\pi}{2} \right) \, \exists C, R > 0 \, \forall z \in S_{\varepsilon}, \, \abs{z} > R: \, 
%  \abs{I_{\alpha}(z) - \frac{\exp^z}{\sqrt{2\pi z}}} \leq \frac{C}{\abs{z}} 
% \] 
% (\cite[\S 10.40.1, 10.17.1]{nist2019}). 
 It follows that for given $n \in \N_0$ we have %\alpha \in \C, \, eventuell unten wieder einfügen, aber ich denke nicht, dass wir das brauchen. War quatsch. 
 \[
  \forall \varepsilon \in \left( 0 , \frac{\pi}{2} \right) \, \exists C, R > 0, \, D \in \C \, \forall z \in S_{\varepsilon}, \abs{z} > R: \, 
  \abs{I^{(n)}_{\alpha}(z) - \frac{D \cdot \exp^z}{\sqrt{2\pi z}}} \leq \frac{C}{\abs{z}}. 
 \]
 From $z \in G$ we get
 \[
  \Re z < d < \frac{c}{1 + \tan ( \frac{\pi-\omega}{2} - \delta) \cdot \tan (\frac{\omega}{2} + \delta)} \quad \text{and} \quad \abs{\arg z} < \frac{\pi - \omega}{2} - \delta. 
 \]
 By Lemma \ref{lemma:properties_of_vw}, the function $[\lambda\mapsto \exp^{-c\lambda}]$ anchors $v_z$ and $w_z$ for all $z \in G$, and, moreover, also all 
 derivatives $\partial^n_z v_z$ and $\partial^n_z w_z$ for $n\in\N$.  
 To sum up we have 
 \[
  \forall n \in \N_0, \, z \in G: \, \partial^n_z v_z, \, \partial^n_z w_z \in \Mer_{\sqrt{A}}. 
 \]
 Since $x \in \mathcal{D} \left( \exp^{c\sqrt{A}} \right)$ it especially follows that $x \in \mathcal{D} \bigl( v_z(\sqrt{A}) \bigr)$ and analogously for $w_z$ and $y$. 
 One concludes that
 \[
  v_z(\sqrt{A})x = \Bigl[ \lambda \mapsto \exp^{-c \lambda} \bigl( v_z(\lambda) - v_z(0) \bigr) \Bigr](\sqrt{A}) \exp^{c \sqrt{A}}x + v_z(0)x
 \]
 which implies the differentiability of the function $z \mapsto v_z(\sqrt{A})x$. 
 Using dominated convergence (pay attention that the given $c>0$ works uniformly for all $z \in G$) one finds
 \[
  \partial_z^n \left( v_z(\sqrt{A})x \right) = \left( \partial^n_z v_z \right)(\sqrt{A})x
 \]
 and analogously for $w$ and $y$. 
 The fact that everything works also in the topology of $\mathcal{D}(A^{\infty})$ is due to
 \[
  \forall k \in \N_0: \, A^k x = A^k \exp^{-c\sqrt{A}} \tilde{x} = \exp^{-\tilde{c} \sqrt{A}} A^k \exp^{-(c-\tilde{c}) \sqrt{A}} \tilde{x}
 \]
 for some $\tilde{c}$ such that $d< \tilde{c} < c$ and $\tilde{x} \in X$. 
 So $x \in \mathcal{D}(A^k)$ and $A^k x \in \mathcal{D} \left( \exp^{\tilde{c} \sqrt{A}} \right)$. 
 Hence, 
 \[
  A^k v_z(\sqrt{A})x = \Bigl[\lambda \mapsto \exp^{-\tilde{c} \lambda} \bigl( v_z(\lambda) - v_z(0) \bigr) \Bigr](\sqrt{A}) \exp^{\tilde{c} \sqrt{A}} A^k x + v_z(0) A^k x
 \]
 implies differentiability in the topology of $\mathcal{D}(A^k)$. 
 Moreover, we obtain
 \[
  \lim\limits_{z \to 0} v_z(\sqrt{A})x = x \; \text{ in } \mathcal{D}(A^{\infty}), 
 \]
 by dominated convergence. 
 That the differential equation is fulfilled follows from the fact that the used functions solve it in the scalar valued case (see for what is written below Remark 
 \ref{remark:better_regularity}) and by what was said above. 
 So,
 \begin{align*}
  & u''(z) + \frac{1-2\alpha}{z}u'(z) - Au(z) \\ = \; & \Bigl[ \lambda \mapsto 
  \Bigl( \underbrace{\partial_z^2 v_z(\lambda) + \frac{1-2\alpha}{z} \partial_z v_z(\lambda) - \lambda^2 v_z(\lambda)}_{= 0} \Bigr) \Bigr](\sqrt{A})x \\ 
  + \; & \Bigl[ \lambda \mapsto 
  \Bigl( \underbrace{\partial_z^2 w_z(\lambda) + \frac{1-2\alpha}{z} \partial_z w_z(\lambda) - \lambda^2 w_z(\lambda)}_{= 0} \Bigr)\Bigr](\sqrt{A})y \\
  = \; & 0. 
 \end{align*}
 That the remaining three limits do act in the right way follows similarly as for $v_z$. 
\end{proof}

In the next step we show uniqueness of the above solution under the additional hypothesis that $A$ is a bounded operator. 

\begin{lemma}
  Let $\omega \in [0,\pi)$, $A \in \Sec_\omega(X)$ be bounded, and $x,y \in X$. Then 
  \begin{enumerate}[label=(\alph*)]\itemsep3pt
    \item
  \begin{align*}
      u''(t) + \frac{1-2\alpha}{t}u'(t) & = Au(t) \quad \bigl(t \in (0,\infty) \bigr),\\
      u(0) & = x,\\
      \lim_{t\to 0\rlim} t^{1-2\alpha} u'(t) & = y
  \end{align*}
  has a unique solution $u\in C \bigl( [0, \infty); X \bigr) \cap C^{\infty} \bigl( (0, \infty); X \bigr)$.
    \item
      Let additionally $T>0$ and $f\from (0,T)\to X$ be such that $[s\mapsto s^{1-2\alpha} f(s)] \in L^1 \bigl((0,T);X\bigr)$.
      Then
        \begin{align*}
      u''(t) + \frac{1-2\alpha}{t}u'(t) & = Au(t) + f(t) \quad \bigl(t \in (0,T) \bigr),\\
      u(0) & = x,\\
      \lim_{t\to 0\rlim} t^{1-2\alpha} u'(t) & = y
  \end{align*}
  has a unique solution $u\in C \bigl( [0, T); X \bigr) \cap C^{\infty} \bigl( (0, T); X \bigr)$ given by
  \[u(t) = v_t(\sqrt{A})x + w_t(\sqrt{A})y + \int\limits_{0}^{t} \left( w_t(\sqrt{A})v_s(\sqrt{A}) - v_t(\sqrt{A})w_s(\sqrt{A}) \right) s^{1-2\alpha}f(s) \d s.\]
  \end{enumerate}
  \label{lemma:uniqueness_bounded_case}
\end{lemma}

\begin{proof}
 (a)	
 The claim follows essentially from Gr\"{o}nwalls inequality, see e.g.\ \cite[Lemma 2.7]{teschl2011}. 
 If the operator $A$ under consideration is continuous the right hand side of our differential equation is continuous. 
 Integrating the alternative expression of the left hand side mentioned in Remark \ref{remark:better_regularity} twice we obtain
 \[
  u(t) = x + \frac{t^{2\alpha}}{2\alpha}y +\frac{A}{2 \alpha} \int\limits_{0}^t \left( t^{2\alpha} - s^{2\alpha} \right) s^{1-2\alpha} u(s) \d s,
 \]
 which in turn yields the inequality
 \[
  \norm{u(t)} \leq \norm{x} + \frac{t^{2 \Re \alpha}}{2 \abs{\alpha}}\norm{y} + \frac{\norm{A}}{2 \abs{\alpha}} \int\limits_{0}^{t} \abs{t^{2\alpha} - s^{2\alpha}} s^{1 - 2 \Re \alpha} \norm{u(s)} \d s. 
 \]
 By Gr\"{o}nwalls inequality we obtain
 \[
  \norm{u(t)} \leq \left( \norm{x} + \frac{t^{2 \Re \alpha}}{2 \abs{\alpha}}\norm{y} \right) \exp^{\frac{\norm{A}}{2 \abs{\alpha}} \int\limits_{0}^{t} \abs{t^{2\alpha} - s^{2\alpha}} s^{1 - 2 \Re \alpha} \d s}. 
 \]
 Hence $x = y = 0$ implies $u = 0$ which yields uniqueness of the solution $u$. 
 
 (b)
 Uniqueness of a solution to the inhomogeneous problem follows from uniqueness of the corresponding homogeneous problem. %, see Lemma \ref{lemma:uniqueness_bounded_case}. 
 Also, the initial condition for the function itself is not influenced by the additional term. 
 A calculation yields
 \begin{align*}
  u'(t) & = \partial_t v_t(\sqrt{A})x + \partial_t w_t(\sqrt{A})y \\ 
        & + \int\limits_{0}^{t} \left( \partial_t w_t(\sqrt{A})v_s(\sqrt{A}) - \partial_t v_t(\sqrt{A})w_s(\sqrt{A}) \right) s^{1-2\alpha}f(s) \d s.
 \end{align*}
 Multiplying this expression with $t^{1-2\alpha}$ and sending $t \to 0\rlim$ the first line will converge to $y$ as follows from Lemma \ref{lemma:solution_to_modified_problem} 
 while the integral will converge to $0$ by dominated convergence. 
 \par
 It remains to check the ODE. 
 Define
 \[
  u(t) := \int\limits_{0}^{t} \left( w_t(\sqrt{A})v_s(\sqrt{A}) - v_t(\sqrt{A})w_s(\sqrt{A}) \right) s^{1-2\alpha}f(s) \d s.
 \]
 We have to show that $u$ is a solution to the ODE 
 \[
  u''(t) + \frac{1-2\alpha}{t}u'(t) = Au(t) + f(t) \quad \bigl(t \in (0,T) \bigr). 
 \] 
 We already calculated
 \[
  u'(t) = \int\limits_{0}^{t} \left( \partial_t w_t(\sqrt{A})v_s(\sqrt{A}) - \partial_t v_t(\sqrt{A})w_s(\sqrt{A}) \right) s^{1-2\alpha}f(s) \d s. 
 \]
 For the second derivative one finds
 \begin{align*}
  u''(t) & = \int\limits_{0}^{t} \left( \partial_t^2 w_t(\sqrt{A})v_s(\sqrt{A}) - \partial_t^2 v_t(\sqrt{A})w_s(\sqrt{A}) \right) s^{1-2\alpha}f(s) \d s \\
         & +   \left( (\partial_t w_t(\sqrt{A})) v_t(\sqrt{A}) - (\partial_t v_t(\sqrt{A})) w_t(\sqrt{A}) \right) t^{1-2\alpha}f(t). 
 \end{align*}
 Plugging everything into the ODE yields
 \begin{align*}
  & u''(t) + \frac{1-2\alpha}{t} u'(t) \\
  & = Au(t) + \left( (\partial_t w_t(\sqrt{A})) v_t(\sqrt{A}) - (\partial_t v_t(\sqrt{A})) w_t(\sqrt{A}) \right) t^{1-2\alpha}f(t),
 \end{align*}
 where we used that $[t \mapsto v_t(\sqrt{A})]$ and $[t \mapsto w_t(\sqrt{A})]$ both solve the ODE on the operator-valued level. 
 In order to see that the remaining summand is really just $f(t)$ one can use 
 \cite[\S 10.28.1]{nist2019} which says
 \[
  I_{-\alpha}(z) I_{\alpha-1}(z) - I_{-\alpha+1}(z) I_{\alpha}(z) = \frac{2 \sin(\alpha \pi)}{z \pi} 
 \]
 and hence, for any $\lambda \in S_{\omega/2 + \delta}$, $\delta > 0$ sufficiently small,
 \begin{align*}
  & t^{1-2\alpha} \bigl(\partial_t w_t(\lambda) \bigr) v_t(\lambda) - t^{1-2\alpha} \bigl(\partial_t v_t(\lambda)\bigr) w_t(\lambda) \\
  & = \frac{\Gamma(\alpha) \Gamma(1-\alpha)}{2} \Bigl( t^{1-2\alpha} t^{\alpha} I_{-\alpha} (\lambda t)\frac{\d}{\d t}  t^{\alpha} I_{\alpha}(\lambda t)
  - t^{1-2\alpha} t^{\alpha} I_{\alpha} (\lambda t) \frac{\d}{\d t}  t^{\alpha} I_{-\alpha}(\lambda t) \Bigr) \\
  & = \frac{\pi t^{1-\alpha}}{2 \sin(\alpha \pi)} \Bigl( I_{-\alpha}(\lambda t) 
  \underbrace{\frac{\d}{\d t}  t^{\alpha} I_{\alpha}(\lambda t)}_{\lambda t^{\alpha} I_{\alpha-1}(\lambda t)}
  - I_{\alpha}(\lambda t) \underbrace{ \frac{\d}{\d t}  t^{\alpha} I_{-\alpha}(\lambda t)}_{\lambda t^{\alpha} I_{-\alpha+1}(\lambda t)} \Bigr) \\
  & = 1,
 \end{align*}
 which transfers on the operator level as
 \[
  t^{1-2\alpha} \bigl(\partial_t w_t(\sqrt{A}) \bigr) v_t(\sqrt{A}) - t^{1-2\alpha} \bigl(\partial_t v_t(\sqrt{A})\bigr) w_t(\sqrt{A}) = \mathop{I}.
 \]
 This finishes the proof.
\end{proof}

% Via approximation this result carries over to the unbounded case. 
% But before, as preparation, we need the fact that since we can solve the homogeneous problem in the bounded case uniquely so we can do with the inhomogeneous problem.

% \begin{lemma}
%  Assume we are still in the situation of Lemma \ref{lemma:uniqueness_bounded_case} and let $f: (0, \infty ) \to X$ be such that $s^{1-2\alpha} f \in L^1_{\text{loc}} \bigl( [0, c);X \bigr)$. 
%  Then the unique solution to the problem of finding $u \in C \bigl( [0, c); X \bigr) \cap C^1 \bigl( (0, c); X \bigr)$ with initial data $u(0)=x$, 
%  and $\lim_{t \to 0\rlim} t^{1-2\alpha}u'(t)=y$
%  \[
%   u''(t) + \frac{1-2\alpha}{t}u'(t) = Au(t) + f(t) \quad (t \in (0,c) ). 
%  \] 
%  is given by
%  \[
%   u(t) = v_t(\sqrt{A})x + w_t(\sqrt{A})y + \int\limits_{0}^{t} \left( w_t(\sqrt{A})v_s(\sqrt{A}) - v_t(\sqrt{A})w_s(\sqrt{A}) \right) s^{1-2\alpha}f(s) \d s. 
%  \]
% \end{lemma}  
% 
% \begin{proof}
%  Uniqueness of a solution to the inhomogeneous problem follows from uniqueness of the corresponding homogeneous problem, see Lemma \ref{lemma:uniqueness_bounded_case}. 
%  Also the initial condition for the function itself is not influenced by the additional term. 
%  A short calculation reveals that this also holds true for the initial condition concerning the scaled derivative. 
%  In the same manner one checks the ODE. 
% \end{proof}

\begin{remark} \leavevmode
\begin{enumerate}
 \item 
  The formula for the solution of the inhomogeneous problem is indeed just the variation of constants formula applied to the equivalent first-order system for the 
  ordinary differential equation, introducing the new variable $w$ given by $w(t):= t^{1-2\alpha}u'(t)$. 
 
 \item
  For a bounded operator $A$, the sectoriality is actually not needed. 
  This is due to the fact that the functions $\lambda \mapsto v_t(\sqrt{\lambda})$ and $\lambda \mapsto w_t(\sqrt{\lambda})$ are actually entire, cf. the beginning of the 
  proof of Lemma \ref{lemma:properties_of_vw} where one can find the power series representations. 
  One can simply plug $A$ into these representations. 
  
\end{enumerate}
\end{remark}

We now extend the uniqueness result to unbounded operators.

\begin{lemma}
 Let $\omega \in [0, \pi)$,  $A \in \Sec _{\omega}(X)$. 
  Let $c>0$ and $x,y \in \mathcal{D}(\exp^{cA})$. 
  Then for $\delta \in (0, (\pi-\omega)/2)$ and $0 < d < c / (1 + \tan((\pi - \omega)/2 - \delta) \tan(\omega/2 + \delta))$ the problem
  \begin{align*}
      u''(t) + \frac{1-2\alpha}{t}u'(t) & = Au(t) \quad (t \in (0,d)),\\
      u(0) & = x,\\
      \lim_{t\to 0\rlim} t^{1-2\alpha} u'(t) & = y
  \end{align*}
  has a unique solution $u\in C \bigl( [0, d); \mathcal{D}(A^{\infty}) \bigr) \cap C^{\infty} \bigl( (0, d); \mathcal{D}(A^{\infty}) \bigr)$ which extends to a holomorphic
  function $u\in \Hol(S_{(\pi-\omega)/2 - \delta} \cap \C_{0< \Re < d}; \mathcal{D}(A^{\infty}))$ such that
  \begin{align*}
    u''(z) + \frac{1-2\alpha}{z}u'(z) & = Au(z) \quad (z \in S_{(\pi-\omega)/2 - \delta} \cap \C_{0< \Re < d}),\\
    \lim\limits_{z \to 0} u(z) & = x,\\
    \lim_{z\to 0\rlim} z^{1-2\alpha} u'(z) & = y. 
  \end{align*} 
 \label{lemma:uniqueness_modified_problem}
\end{lemma}

\begin{proof}
 We work with sectorial approximations, see e.g.\ \cite[Chapter 2.1.2]{haase2006}, and apply the results for bounded operators. 
 For $\varepsilon \in (0, 1]$ consider $A_{\varepsilon} := A(1+\varepsilon A)^{-1}$. 
 Then the family $(A_\varepsilon)_{\varepsilon\in (0,1]}$ in $L(X)$ is a sectorial approximation to $A$. 
 Indeed, the equality  
 \[
  \lambda \bigl( \lambda + A(1+\varepsilon A )^{-1} \bigr)^{-1} = 
  \frac{\lambda}{1 + \lambda \varepsilon} \left( \frac{\lambda}{1+\lambda \varepsilon} + A \right)^{-1} + 
  \frac{\lambda \varepsilon}{1 + \lambda \varepsilon} A \left( \frac{\lambda}{1+\lambda \varepsilon} + A \right)^{-1}
 \]
 implies
 \[
  \sup\limits_{\varepsilon \in (0,1]} \sup\limits_{\lambda > 0} \norm{\lambda \bigl( \lambda + A(1+\varepsilon A )^{-1} \bigr)^{-1}} < \infty
 \]
 and
 \[
  \forall x \in X: \, \lim\limits_{\varepsilon \to 0+} \lambda \bigl( \lambda + A(1+\varepsilon A )^{-1} \bigr)^{-1}x = \lambda (\lambda + A)^{-1}x.  
 \]
 Also, define $u_{\varepsilon}(t) := (1+\varepsilon A)^{-1} u(t)$ and $x_{\varepsilon}:=(1+\varepsilon A)^{-1}x$, $y_{\varepsilon}:=(1+\varepsilon A)^{-1}y$. 
 Then by Lemma \ref{lemma:uniqueness_bounded_case}, the problem
 \begin{align*}
  u_{\varepsilon}''(t) + \frac{1-2\alpha}{t} u_{\varepsilon}'(t) & = A_{\varepsilon} u_{\varepsilon}(t) + \underbrace{A u_{\varepsilon}(t) - 
  A_{\varepsilon} u_{\varepsilon}(t)}_{=:f_{\varepsilon}(t)} 
  \quad \bigl(t\in (0,d) \bigr),\\
  u_\varepsilon (0) & = x_\varepsilon,\\
   \lim_{t\to 0\rlim} t^{1-2\alpha} u_{\varepsilon}'(t) & = y_{\varepsilon},
 \end{align*}
 has the unique solution, implicitly
 given by
 \[
  u_{\varepsilon}(t) = v_{t}(\sqrt{A_{\varepsilon}})x_{\varepsilon} + w_{t}(\sqrt{A_{\varepsilon}})y_{\varepsilon} + 
  \int\limits_{0}^{t} U_{A_{\varepsilon}}(t,s) f_{\varepsilon}(s) s^{1-2\alpha} \d s
 \]
 where 
 \[
  U_{A_{\varepsilon}}(t,s) = w_t(\sqrt{A_{\varepsilon}})v_s(\sqrt{A_{\varepsilon}}) - v_t(\sqrt{A_{\varepsilon}})w_s(\sqrt{A_{\varepsilon}}). 
 \] 
 We now let $\varepsilon \to 0 \rlim$. Then $u_{\varepsilon}(t) \to u(t)$ for $t\in (0,d)$, and $x_{\varepsilon}\to x$ and $y_{\varepsilon}\to y$, 
 see \cite[Proposition 2.1.1 c)]{haase2006}. 
 Moreover,
 \[
  f_{\varepsilon}(s) = \bigl( (1+\varepsilon A)^{-1} - (1+\varepsilon A)^{-2} \bigr) Au(s) \to 0. 
 \]
 If we apply the bounded operator $\exp^{- 2 c\sqrt{A_{\varepsilon}}}$ to the above equality a statement analogously to the similar situation for halfplane operators 
 \cite[Lemma 8.6]{isem21} (also pay attention to the comment on p.\ 159) is applicable, i.e., one has
 \[
  \forall f \in \mathcal{E}^{\text{ext}}[S_{\omega}], x \in X: \, \lim\limits_{\varepsilon \to 0\rlim} f(A_{\varepsilon})x = f(A)x. 
 \] 
 So, sending $\varepsilon  \to 0\rlim$ we get
 \begin{align*}
  \exp^{-2c\sqrt{A}} u(t) & = \exp^{-2c \sqrt{A}} \bigl( v_{t}(\sqrt{A})x + w_{t}(\sqrt{A})y \bigr) \\
  & + \lim\limits_{\varepsilon \to 0 \rlim} \int\limits_{0}^{t} \exp^{-2c \sqrt{A_{\varepsilon}}}U_{A_{\varepsilon}}(t,s) f_{\varepsilon}(s) s^{1-2\alpha} \d s.
 \end{align*}
 In order to apply the dominated convergence theorem and see that the remaining limit actually vanishes (which is the desired result) we need to study the integrand. 
 We have
 \begin{align*}
  & \exp^{- 2 c\sqrt{A_{\varepsilon}}} U_{A_{\varepsilon}}(t,s) x \\
  & = 
  \bigl[\lambda\mapsto \exp^{- c\sqrt{\lambda}} w_t(\sqrt{\lambda}) \bigr](A_{\varepsilon}) \bigl[\lambda\mapsto \exp^{- c\sqrt{\lambda}} v_s(\sqrt{\lambda}) \bigr](A_{\varepsilon})x \\
  & -
  \bigl[\lambda\mapsto \exp^{- c\sqrt{\lambda}} w_s(\sqrt{\lambda}) \bigr](A_{\varepsilon}) \bigl[\lambda\mapsto \exp^{- c\sqrt{\lambda}} v_t(\sqrt{\lambda}) \bigr](A_{\varepsilon})x \\
  & \to \bigl[\lambda\mapsto \exp^{- 2 c\sqrt{\lambda}} U_{\lambda}(t,s) \bigr](A) x,
 \end{align*}
 where $U_{\lambda}(t,s) = w_t(\sqrt{\lambda})v_s(\sqrt{\lambda}) - v_t(\sqrt{\lambda})w_s(\sqrt{\lambda})$.
 By injectivity of $\exp^{-2c\sqrt{A}}$ we finally conclude
 \[
  u(t) = v_{t}(\sqrt{A})x + w_{t}(\sqrt{A})y.  
 \]
 The remaining statements about the analytic continuation follows from Lemma \ref{lemma:solution_to_modified_problem}. 
\end{proof}

For $\omega \in [0, \pi)$,  $A \in \Sec _{\omega}(X)$ and $d>0$ we define
\[
 \mathcal{V}_d := \bigcap\limits_{0 < t < d} \mathcal{D}(v_t(\sqrt{A}))\quad \text{and} 
   \quad \mathcal{W}_d := \bigcap\limits_{0 < t < d} \mathcal{D}(w_t(\sqrt{A})).
\]
Furthermore, set $v_0 := 1$ and $w_0 := 0$ and hence one has $v_0(\sqrt{A}) = \mathop{I}$ as well as $w_0(\sqrt{A}) = 0$. 
   
\begin{corollary}
 Let $\omega \in [0, \pi)$,  $A \in \Sec _{\omega}(X)$, $d>0$, $x,y \in X$ and $u$ a solution to 
 \begin{align*}
      u''(t) + \frac{1-2\alpha}{t}u'(t) & = Au(t) \quad (t \in (0,d)),\\
      u(0) & = x,\\
      \lim_{t\to 0\rlim} t^{1-2\alpha} u'(t) & = y. 
  \end{align*}
  Then
  \begin{enumerate}[label=(\alph*)]\itemsep3pt
    \item If $x\in \mathcal{V}_d$ then $y\in \mathcal{W}_d$.
    \item If $y\in \mathcal{W}_d$ then $x\in \mathcal{V}_d$.
    \item If $x\in \mathcal{V}_d$ and $y\in\mathcal{W}_d$ then $u(t) =  v_t(\sqrt{A})x + w_t(\sqrt{A})y$ for $t\in [0,d)$.
  \end{enumerate}
  \label{coro:form_of_solution}
\end{corollary}

\begin{proof}
 For $c > 0$ sufficiently large we define
 \[
  \tilde{u}(t) := \exp^{-c\sqrt{A}}u(t), \quad \tilde{x} := \exp^{-c\sqrt{A}}x, \quad \tilde{y}:= \exp^{-c\sqrt{A}}y.
 \]
 Then
 \[
  \tilde{u}(t) = v_t(\sqrt{A})\tilde{x} + w_t(\sqrt{A})\tilde{y}
 \]
 by Lemma \ref{lemma:uniqueness_modified_problem}. 
 
 Let us now prove (a), the proof of (b) is analogous. Let $0\leq t<d$. Since $x\in\mathcal{V}_d$ we have $v_t(\sqrt{A})\exp^{-c\sqrt{A}} x = \exp^{-c\sqrt{A}}v_t(\sqrt{A})x$.
 Hence, 
 \[w_t(\sqrt{A})\tilde{y} = \tilde{u}(t) - v_t(\sqrt{A}) \tilde{x} = \exp^{-c\sqrt{A}}\bigl(u(t)-v_t(\sqrt{A})x\bigr) \in \mathcal{D}(\exp^{c\sqrt{A}}),\]
 and therefore
 \[u(t)-v_t(\sqrt{A})x = \exp^{c\sqrt{A}} w_t(\sqrt{A}) \exp^{-c\sqrt{A}}y = \bigl[\lambda\mapsto \exp^{c \lambda} w_t(\lambda) \exp^{-c\lambda}\bigr](\sqrt{A}) y.\]
 In particular, $y\in \mathcal{D}(w_t(\sqrt{A}))$. We conclude $y\in \mathcal{W}_d$.
 
 To prove (c) we just note that
 \[u(t)-v_t(\sqrt{A})x = \bigl[\lambda\mapsto \exp^{c \lambda} w_t(\lambda) \exp^{-c\lambda}\bigr](\sqrt{A}) y = w_t(\sqrt{A})y\]
 for all $0\leq t<d$. 
\end{proof}

% \begin{remark}
%  The above proof shows that the (right-open) existence interval is precisely determined by the regularity of the initial data. 
% \end{remark}

With the above results we can proceed to the uniqueness statement. 
We will need a small refinement of \cite[Lemma 6.3.1]{martinez2001}. 

\begin{lemma}
 Let $A \in \Sec_{\omega}(X)$ and $y \in \bigcap_{t \geq 0} \mathcal{D}(\exp^{t\sqrt{A}})$. 
 If there is $C>0$, $\beta > 0$ such that 
 \[
  \norm{\exp^{t\sqrt{A}}y} \leq C (1 + t^{\beta})
 \]
 it follows that $y \in N(A)$. 

 \label{lemma:laplace_implies_nullspace}
\end{lemma}

\begin{proof}
 The following proof is essentially an adaptation of the proof of \cite[Lemma 6.3.1]{martinez2001}. 
 One has $\sigma(\sqrt{A}) \subset \overline{S_{\omega / 2}}$. 
 Let us denote the Laplace transform of $t \mapsto \exp^{t \sqrt{A}}y$ on $S_{\pi/2}$ by $f$.
 Calculations show that
 \begin{equation}
  (\lambda - \sqrt{A}) f(\lambda) = y
  \label{eq:overlap}
 \end{equation}
 on $S_{\pi/2}$ and 
 \begin{equation}
  \norm{\lambda f(\lambda)} \leq \frac{C}{\cos (\varepsilon)} \cdot \left( 1 + \frac{\Gamma(\beta + 1)}{\left( \Re \lambda \right)^{\beta}} \right) 
  \label{eq:bound}
 \end{equation}
 on $S_{\varepsilon}$ for every $\varepsilon \in [0, \pi / 2 )$. 
 As a result
 \[
  g(\lambda) := \begin{cases} \lambda (\lambda - \sqrt{A})^{-1}y & \text{, if}\quad \lambda \in \C \setminus \overline{S_{\frac{\omega}{2}}}, \\ 
                              \lambda f(\lambda) & \text{, if}\quad \lambda \in \C_{\Re > 0},
                \end{cases}
 \]
 consistently defines a meromorphic function on $\C \setminus \{0\}$. Indeed, by \eqref{eq:overlap}, the functions agree on the overlapping regions. 
 Choose $n \in \N$, $n \geq \beta$. 
 Then $\lambda \mapsto \lambda^n g(\lambda)$ admits a holomorphic extension to $\lambda = 0$ by Riemann's theorem on removable singularities.  
 Furthermore this extension takes the value $0$ at $\lambda = 0$ (approach $\lambda = 0$ from the part where $g$ is bounded) which is why 
 $\lambda \mapsto \lambda^{n-1}g(\lambda)$ is still entire. 
 Arguing inductively down to $n = 0$, we find that $g$ admits an extension to the whole complex plane again denoted by $g$ which is entire and still bounded 
 (approach again $\lambda = 0$ from the part where $g$ is bounded) and so $g$ is constant.
 We have
 \[
  \lim\limits_{\stackrel{\abs{\lambda} \to \infty}{\lambda \in (-\infty, 0)}} g(\lambda) = y
 \]
 which means $g = y$, i.e., $y \in N(\sqrt{A}) = N(A)$ (all fractional powers share the same null space \cite[Prop. 3.1.1]{haase2006}). 
\end{proof}

\begin{thm}
 Let $\omega\in [0,\pi)$ and $A\in \Sec_\omega(X)$.
 Let $u$ be a solution of \eqref{ODE} with $u(0) = 0$. 
 Then $u = 0$. In particular, $\eqref{ODE}$ has a unique solution.
 \label{thm:uniqueness_of_cesp_solution}
\end{thm}

\begin{proof}
 By Corollary \ref{coro:form_of_solution} we know that $u$ is of the form $u(t) = w_t(\sqrt{A})y$ for some suitable $y$. 
 We need to show that the boundedness assumption implies $y=0$. 
 Define for $t > 0$ the function $f_t: S_{\pi/2} \to \C$ by
 \[
  f_t(\lambda) := \frac{\lambda^2}{1 + \lambda^2} \Bigl( w_t (\lambda) - \frac{\Gamma(\alpha)}{2^{1-\alpha}\sqrt{2\pi}} t^{\alpha - \frac{1}{2}} \lambda^{-\alpha - \frac{1}{2}} 
                  \exp^{t \lambda} \Bigr). 
 \] 
 We estimate with $\varphi \in (\omega/2, \pi /2)$
 \begin{align*}
       &  \int\limits_0^{\infty} \abs{f_t \left( s \exp^{i \varphi} \right)} \frac{\d s}{s} \\
  \leq &  \frac{1}{\cos(\varphi)} \int\limits_0^{\infty} \frac{s}{1 + s^2} \abs{w_t \left(s \exp^{i \varphi} \right) - 
          \frac{\Gamma(\alpha)}{2^{1-\alpha}\sqrt{2\pi}} t^{\alpha - \frac{1}{2}} \left( s \exp^{i \varphi} \right)^{-\alpha - \frac{1}{2}} 
          \exp^{t s \exp^{i \varphi}}} \d s \\ 
  =    &  \frac{1}{\cos(\varphi)} \int\limits_0^{\frac{R}{t}} \frac{s}{1 + s^2} \abs{w_t \left(s \exp^{i \varphi} \right) - 
          \frac{\Gamma(\alpha)}{2^{1-\alpha}\sqrt{2\pi}} t^{\alpha - \frac{1}{2}} \left( s \exp^{i \varphi} \right)^{-\alpha - \frac{1}{2}} 
          \exp^{t s \exp^{i \varphi}}} \d s  \\  
     + &  \frac{1}{\cos(\varphi)} \int\limits_{\frac{R}{t}}^{\infty} \frac{s}{1 + s^2} \abs{w_t \left(s \exp^{i \varphi} \right) - 
          \frac{\Gamma(\alpha)}{2^{1-\alpha}\sqrt{2\pi}} t^{\alpha - \frac{1}{2}} \left( s \exp^{i \varphi} \right)^{-\alpha - \frac{1}{2}} 
          \exp^{t s \exp^{i \varphi}}} \d s \\
  \leq &  \frac{t^{2 \Re \alpha - 2}}{\cos(\varphi)} \int\limits_{0}^{R} \frac{s}{1 + \frac{s^2}{t^2}} \abs{w_1 \left(s \exp^{i \varphi} \right) -  
          \frac{\Gamma(\alpha)}{2^{1-\alpha}\sqrt{2\pi}} \left( s \exp^{i \varphi} \right)^{-\alpha - \frac{1}{2}} 
          \exp^{s \exp^{i \varphi}}} \d s \\ 
     + &  \frac{\abs{\Gamma(\alpha)} \exp^{\abs{\Im \alpha} \pi}}{2^{1 - \Re \alpha} \cos(\varphi)} 
          \int\limits_{\frac{R}{t}}^{\infty} \frac{s}{1 + s^2} \frac{t^{\Re \alpha}}{s^{\Re \alpha}} 
          \frac{C}{st} \d s \\
    =: &  (\ast),         
 \end{align*}  
 where we substituted in the first integral and used the asymptotics analogously to the ones derived in Equation \eqref{eq:asymptotics_f1s} for the second integral. 
 Estimating the expression once more results now in
 \begin{align*}
       &  (\ast) \\
  \leq &  \frac{t^{2 \Re \alpha - 2}}{\cos(\varphi)} \int\limits_{0}^{R} s \left( \abs{w_1 \left(s \exp^{i \varphi} \right)} +  
          \frac{\abs{\Gamma(\alpha)}}{2^{1-\Re \alpha}\sqrt{2\pi}} \exp^{\abs{\Im \alpha} \pi} s^{- \Re \alpha - \frac{1}{2}} 
          \exp^{s \cos(\varphi)} \right) \d s \\ 
     + &  t^{\Re \alpha - 1} \frac{C \abs{\Gamma(\alpha)} \exp^{\abs{\Im \alpha} \pi}}{2^{1 - \Re \alpha} \cos(\varphi)} 
          \int\limits_{0}^{\infty} \frac{s^{-\Re \alpha}}{1 + s^2} \d s.         
 \end{align*}  
 It follows that $f_t \in \mathcal{E}(S_{\omega/2})$ and that the entire expression converges to $0$ as $t \to \infty$. 
 Since $u$ is bounded so is the function $t \mapsto A \left( 1 + A \right)^{-1} w_t (\sqrt{A})y$. 
 By the same reasoning as in Lemma \ref{lemma:nested_domains} we have
 \[
    f_t(\sqrt{A})y = A \left( 1 + A \right)^{-1} w_t (\sqrt{A})y - D_{\alpha} t^{\alpha - \frac{1}{2}} \exp^{t \sqrt{A}} 
    A^{\frac{3}{4} - \frac{\alpha}{2}} \left( 1 + A \right)^{-1} y. 
 \]
 where $D_{\alpha} = \Gamma(\alpha) / 2^{1-\alpha}\sqrt{2\pi}$. 
 Since the left hand side converges to $0$ and the first term on the right hand side is bounded, it follows that
 \[
   t \mapsto \exp^{t \sqrt{A}} A^{\frac{3}{4} - \frac{\alpha}{2}} \left( 1 + A \right)^{-1} y
 \] 
 is either bounded ($\Re \alpha \geq 1/2$) or polynomial bounded ($\Re \alpha < 1/2$) for big values of $t$. 
 So one has for some constant $C> 0$
 \[
  \norm{\exp^{t \sqrt{A}} A^{\frac{3}{4} - \frac{\alpha}{2}} \left( 1 + A \right)^{-1} y} \leq C \left(1 + t^{\frac{1}{2}-\Re \alpha} \right)
 \]
 which, by an application of Lemma \ref{lemma:laplace_implies_nullspace}, yields
 \[
  A^{\frac{3}{4} - \frac{\alpha}{2}} \left( 1 + A \right)^{-1} y \in N(A).  
 \] 
 Using again the fact that all fractional powers have the same null space, this implies
 \[
  0 = A^{\frac{\alpha}{2}+\frac{1}{4}} (1+A) A^{\frac{3}{4} - \frac{\alpha}{2}} \left( 1 + A \right)^{-1} y = Ay. 
 \] 
 Thus we can conclude $y \in N(A)$.
 Let know $c>t>0$.  
 It follows that
 \[
  w_t(\sqrt{A})y = \exp^{cA} \frac{1}{2\pi i} \int\limits_{\gamma} w_t(\sqrt{\lambda}) \exp^{-c \lambda} \underbrace{(\lambda-A)^{-1}y}_{= \frac{1}{\lambda}y} \d \lambda
  = w_t(0) \underbrace{\exp^{cA} y}_{= y} = \frac{t^{2\alpha}}{2\alpha} y
 \]
 which is still assumed to be a bounded function in the variable $t$. 
 This finally implies $y = 0$ which was to be shown. 
\end{proof}

\section{Dirichlet-to-Neumann operator}

By Theorem \ref{thm:existence_of_cesp_solution} and Theorem \ref{thm:uniqueness_of_cesp_solution} we know that \eqref{ODE} has a unique solution for all $\alpha \in \C_{0< \Re < 1}$ and all densely defined sectorial operators $A$ in a Banach space $X$
(actually we can drop the assumption of dense domain if we restricted ourselves to the subspace $\overline{\mathcal{D}(A)}$ and studied the part of $A$ in this subspace, see also \cite{meichsner2017}). 
Using this solution the following definition makes sense now.

\begin{defi}
  We define the \emph{(generalised) Dirichlet-to-Neumann operator} $T_{\alpha}$ as  
  \begin{align*}
    \mathcal{D}\left( T_{\alpha} \right) & := \left\{ x \in X \; \middle| \;  u \text{ solves \eqref{ODE} with } u(0)=x,\,\lim\limits_{t \rightarrow 0+} - t^{1-2\alpha}u'(t) \text{ exists} \right\}, \\
    T_{\alpha} x &: = \lim\limits_{t \rightarrow 0+} - t^{1-2\alpha}u'(t).
  \end{align*}
\end{defi}

\begin{thm}
 Let $A$ be a densely defined sectorial operator in $X$. 
 Then $T_{\alpha} = c_{\alpha} A^{\alpha}$ with $c_{\alpha} = \frac{\Gamma(1-\alpha)}{2^{2\alpha-1} \Gamma(\alpha)}$.  
 \label{thm:correspondence_DtN_frac_power}
\end{thm}	

\begin{proof}
 The unique solution $u$ of \eqref{ODE} with initial value $u(0) = x$ is given by $u(z) = u_z(\sqrt{A})x$. 
 For the proof let us label again $u_z$ with the current value of $\alpha$ and write $u_{z, \alpha}$.
 We have $\alpha \in \C_{0 < \Re < 1}$. 
 So $\alpha - 1 \in \C_{-1< \Re < 0}$. 
 Let $x \in X$ be given.
 In the proof of Lemma \ref{lemma:properties_of_u} \ref{4th_point} we already argued
 \[
  - z ^{1-2\alpha} \frac{\d}{\d z} u_{z,\alpha}(\lambda) = c_{\alpha} \lambda^{2\alpha} u_{z,\alpha-1}(\lambda).  			
 \]  
 Plugging in $\sqrt{A}$ for $\lambda$ and regularising $x$ with $(1 + A)^{-\alpha}$ yields 
 \begin{align*}
  -z^{1-2\alpha} \partial_z u_{z, \alpha} (\sqrt{A}) (1+A)^{-\alpha}x & = -z^{1-2\alpha} \left( \partial_z u_{z, \alpha}(\lambda)\right) (\sqrt{A}) (1+A)^{-\alpha}x \\ 
  & = c_{\alpha} \left( \lambda^{2\alpha} u_{z, \alpha - 1}(\lambda) \right)(\sqrt{A}) (1+A)^{-\alpha}x \\ 
  & = c_{\alpha} u_{z, \alpha - 1}(\sqrt{A}) A^{\alpha}(1+A)^{-\alpha}x. 
 \end{align*}
 Now Lemma \ref{lemma:U_is_strongly_continuous} is applicable giving us
 \[
  \lim\limits_{\stackrel{z \to 0}{z \in S_{\delta}}} -z^{1-2\alpha} \partial_z u_{z, \alpha} (\sqrt{A}) (1+A)^{-\alpha}x = c_{\alpha} A^{\alpha}(1+A)^{-\alpha}x
 \]
 which finishes the proof. 
\end{proof}
	
\begin{remark} \leavevmode
   In \cite{meichsner2017} the authors already conjectured that the generalised Dirichlet-to-Neumann operator is already closed. 
   The above theorem proves it. 
\end{remark}

\section{Example and outlook}
\label{sec:example}

We finish the paper with a standard example and give an outlook on open questions. 

\begin{example}[Fractional Laplacian in $L^p$]
 Let $n \geq 2$, $\Omega \subset \R^n$ be a bounded Lipschitz domain fulfilling a uniform outer ball condition and $1 < p < \infty$. 
 We define the weak Dirichlet-Laplacian $A$ by
 \[
  \mathcal{D}(A) := \left\{ u \in W^{1,p}_0(\Omega) \mid \Delta u \in L^p(\Omega) \right\}, \quad Au := -\Delta u. 
 \]
 The so defined operator is the generator of a $C_0$-semigroup of contractions \cite[Theorem 3.8]{wood2007}. 
 In particular, $A$ is sectorial. 
 The generated semigroup is even analytic and the spectrum $p$-independent \cite[Corollary 4.2 and Corollary 4.3]{wood2007}. 
 Our uniquess result Theorem \ref{thm:uniqueness_of_cesp_solution} makes it possible to state a solution for Problem \eqref{ODE} in terms of the semigroup generated by $A$. 
 Namely, it holds that the solution to
 \begin{align*}
  \partial_z^2 u(z,x) + \frac{1-2\alpha}{z} \partial_z u(z,x) & =  - \Delta_x u (z,x) \quad & \bigl((z,x) \in S_{\frac{\pi}{4}} \times \Omega\bigr), \\
  u(0, \cdot ) & = f \in L^p(\Omega),
 \end{align*}
 is given by
 \[
  u(z,x) = \frac{1}{\Gamma(\alpha)} \left( \frac{z}{2} \right)^{2\alpha} \int\limits_0^\infty r^{-\alpha} \exp^{-\frac{z^2}{4r}} \left( \exp^{r \Delta_x}f\right)(x) \, 
  \frac{\d r}{r},
 \]
 see for example \cite[Theorem 2.1]{gale2013}. 
 The operator $A$ admits bounded imaginary powers, see for example \cite[Theorem 2]{duong1996}. 
 Hence, its domain is a complex interpolation space, i.e.\
 \[
  \forall \theta \in (0,1), \alpha \in \C_{\Re > 0}: \, \mathcal{D}(A^{\theta \alpha}) = \left[ L^p(\Omega), \mathcal{D}(A^{\alpha}) \right]_{\theta}. 
 \]
 In particular, for $1 < p \leq 2$ and $\alpha = 1$ one has
 \[
  \mathcal{D}(A) = W^{1,p}_0(\Omega) \cap W^{2,p}(\Omega)
 \]
 \cite[Corollary 3.10]{wood2007} and
 \[
  \mathcal{D}(A^{\theta}) = \left[ L^p(\Omega), W^{1,p}_0(\Omega) \cap W^{2,p}(\Omega) \right]_{\theta}.
  %= \left[ L^p(\Omega), W^{1,p}_0(\Omega) \right]_{\theta} \cap \Bigl[ L^p(\Omega), W^{2,p}(\Omega) \Bigr]_{\theta}. 
 \]
 Further, for $\theta \in (0,1)$, $\alpha \in \C_{\Re > 0}$ such that $\theta \cdot \Re \alpha \in (0,1)$ we have
 \[
  f \in \left[ L^p(\Omega), \mathcal{D}(A^{\alpha}) \right]_{\theta} \, \Leftrightarrow \, u(0, \cdot ) = f \, \text{ and } \, 
  \lim\limits_{z \to 0} z ^{1-2\theta \alpha} \partial_z u(z, \cdot) \, \text{exists in } L^p(\Omega).  
 \]
\end{example}

Concerning the outlook, in \cite{balakrishnan1960} Balakrishnan showed even more. 
\begin{thm}
 Assume $A$ to be a densely defined operator with non-empty resolvent set $\rho(A)$. 
 Further assume that the problem
 \[
  u''(t) = Au(t) \quad (t > 0), \, u(0)=x
 \]
 has a unique bounded solution for all $x \in \mathcal{D}(A)$. 
 Then 
 \[
  u(t) =: \exp^{-Bt}x
 \]
 defines a $C_0$-semigroup whose generator $B$ fulfils $B^2 = A$. 
 If the semigroup is analytic it holds that $A$ is sectorial. 
\end{thm}

The authors conjecture that a corresponding result holds true also for the ODE in \eqref{ODE}.
%Its proof might be subject to another work. 

\bibliographystyle{plainnat}

\bibliography{dtn}

\begin{thebibliography}{28}
\providecommand{\natexlab}[1]{#1}
\providecommand{\url}[1]{\texttt{#1}}
\expandafter\ifx\csname urlstyle\endcsname\relax
  \providecommand{\doi}[1]{doi: #1}\else
  \providecommand{\doi}{doi: \begingroup \urlstyle{rm}\Url}\fi

\bibitem[nis()]{nist2019}
{NIST Digital Library of Mathematical Functions}.
\newblock Version 1.0.22, visited 09.05.2019 \url{http://dlmf.nist.gov/}.

\bibitem[Arendt et~al.(2018)Arendt, ter Elst, and Warma]{arendt2016}
W.~Arendt, A.~F.~M. ter Elst, and M.~Warma.
\newblock {Fractional powers of sectorial operators via the
  Dirichlet-to-Neumann operator}.
\newblock \emph{Comm. Partial Differential Equations}, 43\penalty0
  (1):\penalty0 1--24, 2018.

\bibitem[Balakrishnan(1959)]{balakrishnan1959}
A.~V. Balakrishnan.
\newblock {An operational calculus for infinitesimal generators of semigroups}.
\newblock \emph{Trans. Amer. Math. Soc.}, 91:\penalty0 330--353, 1959.

\bibitem[Balakrishnan(1960)]{balakrishnan1960}
A.~V. Balakrishnan.
\newblock Fractional powers of closed operators and the semigroups generated by
  them.
\newblock \emph{Pacific J. Math.}, 10\penalty0 (2):\penalty0 419--437, 1960.

\bibitem[Bochner(1949)]{bochner1949}
S.~Bochner.
\newblock {Diffusion Equation and Stochastic Processes}.
\newblock \emph{Proc. Nat. Acad. Sciences}, 35:\penalty0 368--370, 1949.

\bibitem[Caffarelli and Silvestre(2007)]{caffarelli2007}
L.~Caffarelli and L.~Silvestre.
\newblock An extension problem related to the fractional laplacian.
\newblock \emph{Comm. Partial Differential Equations}, 32\penalty0
  (8):\penalty0 1245--1260, 2007.

\bibitem[deLaubenfels(1993)]{delaubenfels1993}
R.~deLaubenfels.
\newblock {Unbounded holomorphic functional calculus and abstract Cauchy
  problems for operators with polynomial bounded resolvents}.
\newblock \emph{J. Funct. Anal.}, 114\penalty0 (2):\penalty0 348--394, 1993.

\bibitem[Duong and McIntosh(1996)]{duong1996}
X.T. Duong and A.~McIntosh.
\newblock {Functional Calculi of Second-Order Elliptic Partial Differential
  Operators with Bounded Measurable Coefficients}.
\newblock \emph{J. Geom. Anal.}, 6\penalty0 (2):\penalty0 181--205, 1996.

\bibitem[Engel and Nagel(2000)]{EngelNagel2000}
K.-J. Engel and R.~Nagel.
\newblock \emph{One-parameter semigroups for linear evolution equations},
  volume 194 of \emph{Grad. Texts in Math.}
\newblock Springer, 2000.

\bibitem[Feller(1952)]{feller1952}
W.~Feller.
\newblock On a generalization of marcel riesz' potentials and the semi-groups
  generated by them.
\newblock \emph{Comm. Sém. Mathém. Université de Lund}, pages 73--81, 1952.

\bibitem[Gal{\'e} et~al.(2013)Gal{\'e}, Miana, and Stinga]{gale2013}
J.~E. Gal{\'e}, P.~J. Miana, and P.~R. Stinga.
\newblock Extension problem and fractional operators: semigroups and wave
  equations.
\newblock \emph{J. Evol. Equ.}, 13\penalty0 (2):\penalty0 343--368, 2013.

\bibitem[Haase(2006)]{haase2006}
M.~Haase.
\newblock \emph{The Functional Calculus for Sectorial Operators}.
\newblock Operator Theory: Advances and Applications. Birkh{\"a}user Basel,
  2006.

\bibitem[Haase(2018)]{isem21}
M.~Haase.
\newblock {Lectures on Functional Calculus}, 2018.
\newblock Lecture Notes of the 21st Internet Seminar
  \url{https://www.math.uni-kiel.de/isem21/en/course/phase1/isem21-lectures-on-functional-calculus}.

\bibitem[Hille(1948)]{hille1948}
E.~Hille.
\newblock \emph{Functional analysis and semigroups}, volume~31 of \emph{Amer.
  Math. Soc. Colloquium Publications}.
\newblock de Gruyter, 1948.

\bibitem[Kato(1960)]{kato1960}
T.~Kato.
\newblock {Note on fractional powers of linear operators}.
\newblock \emph{Proc. Japan Acad.}, 36:\penalty0 94--96, 1960.

\bibitem[Kato(1980)]{kato1980}
T.~Kato.
\newblock \emph{Perturbation Theory for Linear Operators}, volume 132 of
  \emph{Grundlagen der mathematischen Wissenschaften}.
\newblock Springer, 2nd edition, 1980.

\bibitem[Komatsu(1969)]{komatsu1969}
H.~Komatsu.
\newblock {Fractional powers of operators, III Negative powers}.
\newblock \emph{J. Math. Soc. Japan}, 21:\penalty0 205--220, 1969.

\bibitem[Lunardi(1995)]{lunardi1995}
A.~Lunardi.
\newblock \emph{Analytic Semigroups and Optimal Regularity in Parabolic
  Problems}.
\newblock Modern Birkh{\"a}user Classics. Springer Basel, 1995.

\bibitem[Martinez and Sanz(2001)]{martinez2001}
C.~Martinez and M.~Sanz.
\newblock \emph{The Theory of Fractional Powers of Operators}.
\newblock North-Holland Mathematics Studies. Elsevier Science, 2001.

\bibitem[McIntosh(1986)]{mcintosh1986}
A.~McIntosh.
\newblock {Operators which have an $H_{\infty}$ functional calculus}.
\newblock \emph{Miniconference on operator theory and partial differential
  equations}, pages 210--231, 1986.

\bibitem[Meichsner and Seifert()]{meichsner2017}
J.~Meichsner and C.~Seifert.
\newblock {Fractional powers of non-negative operators in Banach spaces via the
  Dirichlet-to-Neumann operator}.
\newblock arxiv preprint \url{https://arxiv.org/abs/1704.01876}.

\bibitem[Molchanow and Ostrovskii(1968)]{molchanow1968}
S.~A. Molchanow and E.~Ostrovskii.
\newblock Symmetric stable processes as traces of degenerate diffusion
  processes.
\newblock \emph{Theory Probab. Appl.}, 14\penalty0 (1):\penalty0 128--131,
  1968.

\bibitem[Nelson(1958)]{nelson1958}
E.~Nelson.
\newblock {A Functional Calculus Using Singular Laplace Integrals}.
\newblock \emph{Trans. Amer. Math. Soc.}, 88\penalty0 (2):\penalty0 400--413,
  1958.

\bibitem[Phillips(1952)]{phillips1952}
R.~S. Phillips.
\newblock {On the generation of semigroups of linear operators}.
\newblock \emph{Pacific J. Math.}, 2:\penalty0 343--369, 1952.

\bibitem[Schilling et~al.(2012)Schilling, Song, and Vondracek]{schilling2012}
R.~L. Schilling, R.~Song, and Z.~Vondracek.
\newblock \emph{Bernstein Functions: Theory and Applications}, volume~37 of
  \emph{de Gruyter Stud. Math.}
\newblock de Gruyter, 2nd edition, 2012.

\bibitem[Stinga and Torrea(2010)]{stinga2010}
P.~R. Stinga and J.~L. Torrea.
\newblock {Extension Problem and Harnack's Inequality for Some Fractional
  Operators}.
\newblock \emph{Comm. Partial Differential Equations}, 35\penalty0
  (11):\penalty0 2092--2122, 2010.

\bibitem[Teschl(2011)]{teschl2011}
G.~Teschl.
\newblock \emph{{Ordinary Differential Equations and Dynamical Systems}}.
\newblock Graduate Studies in Mathematics. American Mathematical Society, 10th
  edition, 2011.

\bibitem[Wood(2007)]{wood2007}
I.~Wood.
\newblock {Maximal $L^p$-regularity for the Laplacian on Lipschitz domains}.
\newblock \emph{Math. Z.}, 255\penalty0 (4):\penalty0 855--875, 2007.

\end{thebibliography}

\end{document}